\RequirePackage{silence}
\documentclass[10pt,journal,compsoc]{sty/IEEEtran/IEEEtran}

\newcommand{\ourTitle}{A Divide and Conquer Approximation Algorithm for Partitioning Rectangles}
\newcommand{\ourAuthors}{Reyhaneh Mohammadi and Mehdi Behroozi}
\newcommand{\ourKeywords}{Space Partitioning Optimization; Computational Geometry; Plant Layout; VLSI Design; Treemap Visualization; Soft Rectangle Packing}
\newcommand{\shortitle}{Mohammadi and Behroozi: Approximate Partitioning of Rectangles}

\newcommand{\mailtohref}[1]{\href{mailto:mohammadi.re@northeastern.edu
;m.behroozi@northeastern.edu}{\color{black}#1}}

\newcommand{\ourAuthorsFormatted}{%
    \author{
        Reyhaneh Mohammadi\IEEEauthorrefmark{2} and
        Mehdi Behroozi\IEEEauthorrefmark{1}
	\IEEEcompsocitemizethanks{%
	\IEEEcompsocthanksitem \IEEEauthorrefmark{2}R. Mohammadi is with Northeastern University, Boston, MA, USA. \protect\\
Email: \mailtohref{mohammadi.re@northeastern.edu}
\IEEEcompsocthanksitem \IEEEauthorrefmark{1}Corresponding Author: M. Behroozi is with Northeastern University, Boston, MA, USA. \protect\\
Email: \mailtohref{m.behroozi@northeastern.edu} \protect\\
Phone: (617) 373-2032  \protect\\
Address: 449 Snell Engineering Center, 360 Huntington Ave, Boston, MA 02115  \protect\\
M. Behroozi gratefully acknowledges the support of Northeastern University for this research.
	}
    }
}

\usepackage[nocompress]{cite}

\usepackage{graphicx}

\graphicspath{{./figs/}}

\usepackage{url}
\usepackage[bookmarks=true, pdfpagelabels=true, pdftex]{hyperref}
\hypersetup{
    bookmarksnumbered=true,
    pdfpagemode={UseOutlines},
    plainpages=false,
    colorlinks=true,
    linkcolor={black},
    citecolor={black},
    urlcolor={black},
    pdftitle={\ourTitle},
    pdfsubject={Typesetting},
    pdfauthor={\ourAuthors},
    pdfkeywords={\ourKeywords},
}
\usepackage[NewJournal]{sty/newmacros}

\begin{document}

\title{\ourTitle}
\ourAuthorsFormatted
\markboth{\shortitle}%
{fish}

\IEEEtitleabstractindextext{%
\begin{abstract}
    Given a rectangle $R$ with area $A$ and a set of areas $L=\{A_1,...,A_n\}$ with $\sum_{i=1}^n A_i = A$, we consider the problem of partitioning $R$ into $n$ sub-regions $R_1,...,R_n$ with areas $A_1,...,A_n$ in a way that the total perimeter of all sub-regions is minimized. The goal is to create square-like sub-regions, which are often more desired. We propose an efficient $1.203$--approximation algorithm for this problem based on a divide and conquer scheme that runs in $\mathcal{O}(n^2)$ time. For the special case when the aspect ratios of all rectangles are bounded from above by 3, the approximation factor is $2/\sqrt{3} \leq 1.1548$. We also present a modified version of out algorithm as a heuristic that achieves better average and best run times. 
\end{abstract}

}

\maketitle

\IEEEraisesectionheading{\section{Introduction}
\label{sec:introduction}}

\IEEEPARstart{P}{artitioning} of a rectangle into several sub-rectangles while optimizing some partition metric is a well-known geometric optimization problem with many different applications such as plant layout design \cite{anjos2006new,aiello2012multi,xiao2017problem,ji2017iterative,zawidzki2020multi}, geographic resource allocation \cite{aronov2006minimum,carlsson2012dividing,carlsson2013dividing,behroozi2016robust,carlsson2016geometric,behrooziCarlsson2020Springer}, treemapping in data visualization \cite{shneiderman1991tree,shneiderman1992tree,wattenberg1999visualizing,bruls2000squarified,shneiderman2001ordered,bederson2002ordered,engdahl2005ordered}, VLSI Design \cite{pan1996area,stockmeyer1983optimal,gonzalez1989improved,lopez1996efficient,anjos2012global}, and data assignment problem in parallel computers \cite{khanna1998approximating,beaumont2001matrix,lorys2000rectangle,ghosal2020rectangle}. 

This problem is also closely-related to many geometric and space partitioning optimization problems that include packing, covering, and tiling---generally focused on minimizing wasted space or optimally allocating geographical resources, as well as data visualization problems--focusing on finding solutions that are visually appealing. 
These problems include: cutting stock; knapsack; bin packing; guillotine; disk covering; polygon covering; kissing number; strip packing; rectangle packing; square packing; squaring the square; squaring the plane; and, in 3D space, cubing the cube; tetrahedron packing; and treemapping. Our problem is also called soft rectangle packing in the terminology of packing problems. These problems have a large body of literature and a long history that can be perhaps traced back to some geometric problems in the ancient era such as Queen Dido's problem in ancient Carthage \cite{ashbaugh2010problem}.

In this paper, we consider the following problem. Given a rectangle $R$ with $\Area(R)=A$ and a list of areas $A_1,...,A_n$ with $\sum_{i=1}^n A_i = A$ we want to partition $R$ into $n$ sub-rectangles with areas $A_1,...,A_n$ in a way that resulting rectangles are as square as possible. It is a common desire in various fields to have square-like rectangles. For this goal we try to minimize the sum of the perimeters of all sub-rectangles. Our choice of the total (average) perimeter is motivated by the fact that the perimeter of a rectangle is minimized when it is a square. We define the perimeter of $R$ as $\Perim(R) = \width(R) + \height(R)$. We also define the aspect ratio of a rectangle $R$ as
\[
\AR(R)=\max\left\{ \frac{\width(R)}{\height(R)}, \frac{\height(R)}{\width(R)}\right\},
\]
 i.e., $\AR\geq 1$ and the aspect ratio of a square is one.

The NP-hardness of the problem with different objective functions has been proved in several ways. In data visualization field this problem with the goal of minimizing the maximum aspect ratio of all sub-rectangles---was noted as NP-hard by Bruls \etal \cite{bruls2000squarified}.
de Berg \etal later proved the problem is strongly NP-hard with a reduction from the square packing problem \cite{deBerg2014Treemaps}.
The related problem of minimizing the total perimeter of all sub-rectangles was proved by Beaumont \etal \cite{beaumont2001matrix} to be NP-hard, using a reduction from the problem of partitioning a set of integers into two subsets of equal sum.
Given this computational complexity, we  settle with non-exact approaches and develop one approximation algorithm and one heuristic algorithm to find high quality partitions efficiently.

\section{Algorithm}
\label{sec:approxAlg}
In this section, we propose an approximation algorithm based on a divide and conquer scheme. 
Divide \& conquer approach has been previously used for this problem as in \cite{nagamochi2007approximation,fugenschuh2014exact,liang2015divide,behroozi2023treemap}, where the approximation guarantee is provided in the first two works.
The divide \& conquer approach presented by Nagamochi and Abe in \cite{nagamochi2007approximation}, suggests a factor 1.25 approximation algorithm.
F{\"u}genschuh \etal \cite{fugenschuh2014exact} modified this algorithm and achieved better result for some instances and worse on some others. In their analysis, they distinguish slow-decreasing and fast-decreasing sequences of areas. Slow-decreasing sequences refer to the case where the areas are of similar size. For such sequence of areas the approximation ratio of their algorithm is $2/\sqrt{3} \leq 1.1548$. For the faster-decreasing sequences they find an upper bound for the approximation ratio that depends on the decaying rate and is bounded above by 1.7657.
Our algorithm improves these results, which to the best of our knowledge are still the best among the existing algorithms.

\begin{algorithm}[t]
\SetAlgoLined
\BlankLine
    \SetKwFunction{main}{main}
    \SetKwFunction{proc}{Partition}
    \SetKwProg{Fmain}{Function}{}{}
    \SetKwProg{Fproc}{Procedure}{}{}
    
\KwIn{Rectangle $R$ 
and a list of $n$ areas $L=\{A_1,A_2,...,A_n\}$ with $\sum_{i=1}^n A_i = \Area (R)$.}
\KwOut{Partition $R_1,R_2,...,R_n$ with areas $A_1,A_2,...,A_n$.} 
\BlankLine
\tcc{***********************************************}
\Fmain{\main{R,L}}{
 \eIf{$n=1$}{ \Return{$R$}\;}{ 
   Let $w=\width(R)$ and $h=\height(R)$\;
   Sort $L$ in a non-increasing order and reindex the sorted areas as $A_1,A_2,...,A_n$\;
   \KwRet{${\sf Partition}(R,L,1,n)$}\;
  }
}
\setcounter{AlgoLine}{0}
\Fproc{\proc{$Q, \Lambda,\mbox{start, stop}$}}{
   Let $B_{(\cdot)} = \{A_i,...,A_j\}$ denote a block of areas with $\Area(B_{(\cdot)}) = \sum_{k=i}^j A_k$\;
  Set $L^{\prime} = \{B_{(1)},...,B_{(n)}\}$, where $B_{(i)} = \{A_i\},\; \forall i$\;
  Set $L^{\prime\prime} = \{a_1,a_2,...,a_n\}$, where $a_i=\Area(B_{(i)}),\; \forall i$\;
  \While {$\left\vert L^{\prime\prime} \right\vert > 2$}{
    Set $m=\left\vert L^{\prime\prime}\right\vert - 1$\;
    Set $a_m = a_{m}+a_{m+1}$\;
    Insert $a_m$ back into $L^{\prime\prime}$ at location $k$ with $1\leq k \leq m$ that keeps $L^{\prime\prime}$ sorted\;
    Set $B_{(m)} = \cup_{j=m}^{\left\vert L^{\prime\prime} \right\vert} B_{(j)}$ and $L^{\prime} = \{B_{(1)},...,B_{(k-1)},B_{(m)},B_{(k)},...,B_{(m-1)}\}$\;    
    \If{$k \neq m$}{
    Reindex the blocks in $L^{\prime}$ as $B_{(1)},B_{(2)},...,B_{(m)}$\;
    Reindex the areas in $L^{\prime\prime}$ as $a_1,a_2,...,a_m$\;
    }
}
Set $\Lambda_1 = L^{\prime}(1)=B_{(1)}$ and $\Lambda_2 = L^{\prime}(2)=B_{(2)}$\;
Let $w$ denote the width of $Q$ and $h$ be the height\;
\eIf{$w > h$}{
Divide $Q$ with a vertical line into two pieces $Q_1$ and $Q_2$ with area $L^{\prime\prime}(1)=a_1$ on the left and $L^{\prime\prime}(2)=a_2$ on the right\;}{
Divide $Q$ with a horizontal line into two pieces $Q_1$ and $Q_2$ with area $L^{\prime\prime}(1)=a_1$ on the top and $L^{\prime\prime}(2)=a_2$ on the bottom\;}
\KwRet{$\proc(Q_{1},\Lambda_{1}) \cup \proc(Q_{2},\Lambda_{2})$}\;
}
\protect\caption{\label{alg:approxAlg} Algorithm ${\sf ApproximationDC}(R,L)$; it takes as input a rectangle and a list of  areas. It generates a treemap of sub-rectangles with the given areas according to a divide and conquer approach.}
\end{algorithm}

In our \modelApproximateDC algorithm, we first sort the areas in a non-ascending order and then recursively merge the \emph{two smallest areas}, while retaining the list of areas sorted, to finally end up with two compounded areas. Then, we partition $R$ into two segments, horizontally or vertically, with these two compounded areas and apply the algorithm to each of these segments. The pseudocode for \modelApproximateDC, for the case where cuts are either vertical or horizontal (depending on the width \& height of the remaining segment), is shown in \cref{alg:approxAlg}.
It should be noted that 
 it can be generalized to polygonal and angular cuts and has no restriction on input layout container shape. However, our approximation factor analysis here is restricted to the case where the input shape is a rectangle and the cuts are rectangular.
This algorithm achieves computational time of $\mathcal{O}(n^2)$. 

\begin{algorithm}[tbp]
    \SetAlgoLined
    \BlankLine
    \SetKwFunction{main}{main}
    \SetKwFunction{proc}{Partition}
    \SetKwProg{Fmain}{Function}{}{}
    \SetKwProg{Fproc}{Procedure}{}{}
    
    \KwIn{Rectangle $R$ 
    and a list of $n$ areas $L=\{A_1,A_2,...,A_n\}$ with $\sum_{i=1}^n A_i = \Area (R)$.}
    \KwOut{Partition $R_1,R_2,...,R_n$ with areas $A_1,A_2,...,A_n$.} 
    \BlankLine
    \tcc{***********************************************}
\Fmain{\main{R,L}}{
 \eIf{$n=1$}{ \Return{$R$}\;}{ 
   Let $w=\width(R)$ and $h=\height(R)$\;
   Sort $L$ in a non-increasing order and reindex the sorted areas as $A_1,A_2,...,A_n$\;
   \KwRet{${\sf Partition}(R,L,1,n)$}\;
  }
}
\setcounter{AlgoLine}{0}
\Fproc{\proc{$Q, \Lambda,\mbox{start, stop}$}}{
    Let $B_{(\cdot)} = \{A_i,...,A_j\}$ denote a block of areas with $\Area(B_{(\cdot)}) = \sum_{k=i}^j A_k$\;
    Set $L^{\prime} = \{B_{(1)},...,B_{(n)}\}$, where $B_{(i)} = \{A_i\},\; \forall i$\;
    Set $L^{\prime\prime} = \{a_1,a_2,...,a_n\}$, where $a_i=\Area(B_{(i)}),\; \forall i$\;
    \While {$\left\vert L^{\prime\prime} \right\vert > 2$}{
        Set threshold $\tau = \sum_{i=1}^{\left\vert L^{\prime\prime} \right\vert} a_i \,/ \left\vert L^{\prime\prime} \right\vert$\;
        Find the smallest $i$ such that $a_i < \tau$\;
        \eIf{
            $i<\left\vert L^{\prime\prime} \right\vert$}{$m=i$\;
        }{    
            Set $m = \lceil \left\vert L^{\prime\prime} \right\vert/2 \rceil$\;
        }  
        Set $a_m = \sum_{j=m}^{\left\vert L^{\prime\prime} \right\vert} a_j$ and $L^{\prime\prime} = \{a_1,...,a_{m-1}\}$\;
        Insert $a_m$ back into $L^{\prime\prime}$ at location $k$ with $1\leq k \leq m$ that keeps $L^{\prime\prime}$ sorted\;
        Set $B_{(m)} = \cup_{j=m}^{\left\vert L^{\prime\prime} \right\vert} B_{(j)}$ and $L^{\prime} = \{B_{(1)},...,B_{(k-1)},B_{(m)},B_{(k)},...,B_{(m-1)}\}$\;    
        \If{$k \mathrel{\mathtt{!=}} m$}{
            Reindex the areas in $L^{\prime\prime}$ as $a_1,a_2,...,a_m$\;
            Reindex the blocks in $L^{\prime}$ as $B_{(1)},B_{(2)},...,B_{(m)}$\;
        }
    }
    Set $\Lambda_1 = L^{\prime}(1)=B_{(1)}$ and $\Lambda_2 = L^{\prime}(2)=B_{(2)}$\;
    Let $w$ denote the width of $Q$ and $h$ be the height\;
    \eIf{$w > h$}{
        Divide $Q$ with a vertical line into two pieces $Q_1$ and $Q_2$ with area $L^{\prime\prime}(1)$ on the left and $L^{\prime\prime}(2)$ on the right\;
    }{
        Divide $Q$ with a horizontal line into two pieces $Q_1$ and $Q_2$ with area $L^{\prime\prime}(1)$ on the top and $L^{\prime\prime}(2)$ on the bottom\;
    }
    \Return{$\proc(Q_{1},\Lambda_{1}) \cup \proc(Q_{2},\Lambda_{2})$}\;
  }  
    \caption{
        \modelModifiedDC$(R,L)$
        Generates a rectangular treemap with the given areas and bounding rectangle according to a divide and conquer approach.
    }
    \label{alg:ModifiedDivideConquer} 
\end{algorithm}

\subsection{Improving the Best and Average Running Times}
Bundling the two smallest areas, could be modified to reduce $\mathcal{O}(n)$ such operations in the list. This may change the quality of the solution, since we would no longer have the approximation guarantee. However, having a faster alternative of the algorithm could often be useful.
In our \modelModifiedDC algorithm we first sort the areas in a non-ascending order and then recursively merge \emph{all areas below some threshold} to finally end up with two compounded areas. Here, we set the threshold to be the average of considered areas in each iteration.
If at any point we have more than two areas and none of them is below the threshold, we divide the list by half and then sum them to finally end up with two total subareas.
Then, similar to \cref{alg:approxAlg}, we partition $R$ into two segments with these two compounded areas and apply the algorithm to each of these two segments. The worst case of these two algorithms is the same but the best and average performance in \cref{alg:ModifiedDivideConquer} improves. The magnitude of this improvement highly depends on the input list that determines the behavior aroud the set threshold $\tau$ in each iteration and the number of times the input list gets divided by 2 in each recursive call.

The pseudocode for \modelModifiedDC, for the case where cuts are either vertical or horizontal (depending on the width \& height of the remaining segment), is shown in \cref{alg:ModifiedDivideConquer}.
As shown in the divide \& conquer of \cite{liang2015divide}, it can be easily modified and generalized to handle polygonal and angular cuts and to have no restriction on input layout container shape. This is also true for \modelApproximateDC but in this paper we only analyze the approximation guarantee for the case when we have rectangular input region and output sub-regions.

 \begin{figure}[tbp]%
    \centering%
    \begin{subfigure}[b]{.75\textwidth}%
        \centering
        \includegraphics[sqrtofarea=.99\textwidth]{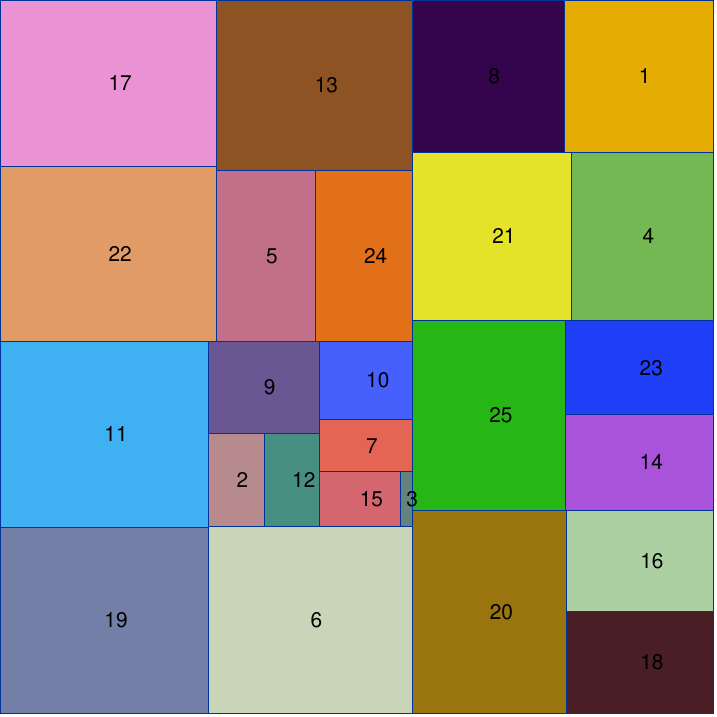}%
        \caption*{ }%
        \label{fig:apx-test-25}%
    \end{subfigure}%
    \\
    \vspace{-16pt}
     \begin{subfigure}[b]{.75\textwidth}%
        \centering
        \includegraphics[sqrtofarea=.99\textwidth]{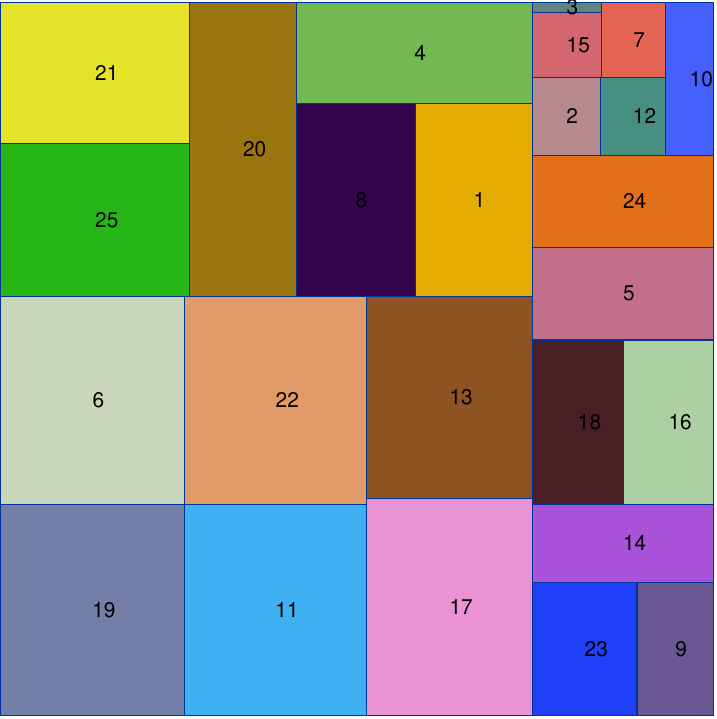}%
        \caption*{ }%
        \label{fig:apx-test-25-MDC}%
    \end{subfigure}%
    \\
    \vspace{-10pt}
    \subfigsCaption{\protect
        Illustration of the output on a random instance with 25 areas by \cref{alg:approxAlg} at the top with total perimeter of 19.0773 
        and \cref{alg:ModifiedDivideConquer} 
        on the bottom with total perimeter of 19.5197. Sub-rectangle follow the same labeling and color coding.
    }%
    \label{fig:APX-MDC}%
\end{figure}

\subsection{Analysis of Approximate Divide and Conquer}
\label{sec:AnalysisApxDC}

Our approach in the analysis of \modelApproximateDC is in the same spirit of \cite{nagamochi2007approximation}, although our algorithm is totally different. We begin with some definitions we use in proving some critical characteristics of our algorithm and then we introduce the lower and upper bounds.

\subsubsection{Critical Characteristics}
\label{sec:ApxCharacteristics}

Let $\mathcal{R}=\{R_1,...,R_n\}$ be a partition of $R$ output by \cref{alg:approxAlg}. Rectangles $R_i$ in this output are called \emph{simple }rectangles. Any intermediary input for some recursive call of \cref{alg:approxAlg} is called a \emph{compound} rectangle. Every call of the algorithm dissects an input region $R$ into two rectangles $R_1$ and $R_2$. For simplicity, we call them the \emph{left child} and the \emph{right child}, respectively, regardless of their actual positioning.

  \begin{lem}
  Algorithm \ref{alg:approxAlg} recursively partitions a rectangle $R$ containing rectangles $R^{\prime}_i,R^{\prime}_{i+1},...,R^{\prime}_k$ with areas $A_i \geq A_{i+1} \geq \cdots \geq A_k$ into two rectangles $R_1$ and $R_2$ such that $\Area(R_1)\geq \frac{1}{3}\Area(R)$ and if $A_i\leq \frac{2}{3}\Area(R)$, then we also have $\Area(R_2)\geq \frac{1}{3}\Area(R)$.    
\end{lem}
\begin{proof}
If $A_i\geq \frac{1}{2}\Area(R)$, its order in the list never changes and the other sub-areas get summed up until only two areas remain in the list, where $\Area(R_1)=A_i$ will be the first one. As a result, in this case, if $\frac{1}{2}\Area(R) \leq \Area(R_1) \leq \frac{2}{3}\Area(R)$, we will have $\Area(R_2)\geq \frac{1}{3}\Area(R)$, otherwise, $\Area(R_2)\leq \frac{1}{3}\Area(R)$.\\
Next we consider the case $\frac{1}{2}\Area(R) > A_i \geq A_{i+1} \geq \cdots \geq A_k$. Right before we end up with the final two children $R_1$ and $R_2$, we have three sub-areas. Let them be $a_1, a_2$ and $a_3$. If $a_1=A_i \leq \frac{1}{2}\Area(R)$, it should be also greater than $\frac{1}{3}\Area(R)$, otherwise it could not be located at the beginning of the list. If $a_1$ is not the area of an original rectangle of the input list, it should be summation of two other sub-areas that have areas less than $a_2$ and $a_3$. So, we have $a_1 \geq a_2 \geq a_3$ and $a_1 \leq a_2+a_3 \leq 2a_1$. Hence, $\Area(R_1)=a_2+a_3$ and $\Area(R_2)=a_1$. 
As a result, we have $\sfrac{a_1}{(a_1+a_2+a_3)} \geq \sfrac{1}{3}$ and thus $\Area(R_1) \geq a_1 \geq \frac{1}{3}\Area(R)$.
\end{proof}
\begin{lem}
\label{lem:ApxChildrenAspectRatios}
Let $R_1$ and $R_2$ be the left and right children of a compound rectangle R, and let $A_{\ell}\geq A_{\ell+1}\geq ...\geq A_{m}$ be the sub-areas of $A(R)$. Then,
\begin{enumerate}[label=(\roman*)]
\item $\AR(R_1)\leq \max\{\AR(R), 3\}$\,,
\item $\AR(R_2)\leq \max\{\AR(R), 3, (1 +\frac{A_{\ell}}{A_{\ell+1}})\}$.
\end{enumerate}
Moreover, if $\AR(R_2) > 3$ and $\AR(R)\leq \AR(R_2)$, then $R_1$ is a simple rectangle with $\Area(R_1)=A_{\ell}>\frac{2}{3}\Area(R)$ and $\AR(R_2)\leq 1+\frac{A_{\ell}}{A_{\ell+1}}$.
\end{lem}
\begin{proof}
The proof is the same as the proof of Lemma 3.1 in \cite{nagamochi2007approximation}.
\end{proof}
\begin{cor}
Given a rectangle $R$ and a list $L=\{A_{\ell},A_{\ell},...,A_m\}$ with $\sum_{A_i \in L} A_i = \Area (R)$ and $A_{\ell}\geq A_{\ell+1}\geq ...\geq A_{m}$, let $\mathcal{R}=\{R_{\ell},R_{\ell+1},...,R_m\}$ be the output of Algorithm ${\sf ApproximationDC}(R,L)$. Let $\mathcal{R^{\prime}}$ be the set of all rectangles that are the input of some recursive call of ${\sf ApproximationDC}(R,L)$. Then, for each rectangle $\hat{R} \in \mathcal{R} \cup \mathcal{R^{\prime}}$ we have
\begin{equation}
\label{eq:ApxARbound}
\AR(\hat{R}) \leq \max\{\AR(R),3,(1+ \max_{A_i\in L} \frac{A_i}{A_{i+1}})\}
\end{equation}
\end{cor}
\begin{proof}
By Lemma \ref{lem:ApxChildrenAspectRatios} we have $\max\{\AR(R_1),\AR(R_2)\} \leq \max\{\AR(R), 3, (1 +\frac{A_{\ell}}{A_{\ell+1}})\}$ for $A_{\ell}=\max_{A_i\in L} A_i$. Then, \eqref{eq:ApxARbound} follows by induction.
\end{proof}

\subsubsection{Lower Bound}
\label{sec:ApxLB}
In order to find lower bounds, we first adopt the definition of forced rectangles from \cite{nagamochi2007approximation}. 
\begin{Def}
\label{def:forcedRect}
The original rectangle $R$ is a forced rectangle. The right child $R_2$ of $R$ is a forced rectangle if $A_1 \geq \Area(R)/2$. Any rectangle whose long edges are both contained in the long edges of a forced rectangle is defined to be a forced rectangle. 
\end{Def}
Note that the parent of any forced rectangle is also a forced rectangle. 
\begin{lem}
\label{lem:ApxLB}
Given a rectangle $R$ and a list $L=\{A_{1},A_{2},...,A_n\}$ with $\sum_{A_i \in L} A_i = \Area (R)$, let $\mathcal{R}=\{R_{1},...,R_n\}$ be a partition of $R$. For any $R_i \in \mathcal{R}$ let $\LB_i$ denote a lower bound on $\Perim(R_i)$. Then,
\begin{enumerate}[label=(\roman*)]
\item for any forced simple rectangle $R_i \in \mathcal{R}$, $\LB_i=\width(R_i) + \height(R_i)$ is the tight lower bound on $\Perim(R_i)$.
\item for any non-forced simple rectangles $R_j \in \mathcal{R}$, we have $\Perim(R_j) \geq \LB_j=2\sqrt{A_j}$. 
\end{enumerate}
\end{lem}
\begin{proof}
The lower bound in \emph{(i)} follows from the fact that the length of a short edge of $R_i$ cannot be larger than $\min\{\height(R_i),\width(R_i)\}$ in any other partitioning scheme of $R$. The lower bound in \emph{(ii)} is trivial.
\end{proof}

\subsubsection{Upper Bound}
\label{sec:ApxUB}
The upper bound in our algorithm highly depends on the aspect ratio of the generated sub-rectangles in the partition. Therefore, we break the analysis into different cases depending on aspect ratios.

~\\
\textbf{Case I $(\bm{\AR(R_i) \leq 3, \; i=1,...,n)}$:} If all sub-rectangles created by our algorithm either have aspect ratios less than 3 or are simple forced rectangles we will have the approximation factor less than $2/\sqrt{3} < 1.1548$. This can be simply shown by the fact that for simple forced rectangles the ratio is 1 and for the other case we have
\[
\resizebox{\columnwidth}{!}{%
$\frac{\sum_{1 \leq i\leq n}\Perim(R_i)}{\sum_{1 \leq i\leq n} \LB_i} \leq \smash{\displaystyle\max_{1 \leq i \leq n}}\{\frac{\Perim(R_i)}{\LB_i}\} \leq \smash{\displaystyle\max_{1 \leq i \leq n}} \{\frac{\sqrt{3A_i}+\sqrt{A_i/3}}{2\sqrt{A_i}}\} = \frac{2}{\sqrt{3}}\,,$
}
\]
where the first inequality comes from the fact that all summands in both numerator and denominator are positive reals.

~\\
\textbf{Case II $(\bm{\AR(R_i) > 3, \; \textrm{ for some } i)}$:}
Suppose at some point in our algorithm we come up with a rectangle with aspect ratio greater than 3. We will find the \emph{worst case scenarios} in terms of approximation factor. For simplicity of notation in the analysis, here we use $\rho$ to denote the aspect ratio.
\begin{lem}
\label{lem:ApxARGe3}
If the algorithm divides $R$ into two rectangles $R_1$ and $R_2$ and $\Area(R_2)=z, \; \width(R_2)= w_z,\; \height(R_2)=h_z$ (with $w_z\geq h_z$), and $\AR(R_2)=\rho_z\geq 3$, which makes the total approximation factor to come above $2/\sqrt{3}$, the approximation factor will be maximized when $R_2$ is further divided into $R_{z1}$ and $R_{z2}$, where $R_{z1}$ is a simple rectangle and $R_{z2}$ has $w_{R_{z2}} \leq h_{R_{z2}}$.
\begin{proof}
We first show, by contradiction, that $R_{z1}$ cannot be a combined rectangle consisting of two sub-rectangles having width greater than height and aspect ratios greater than 3. Assume that $R_{z1}$ is a combined rectangle consisting of two sub-rectangles $R_{z^{\prime}}$ and $R_{z^{\prime\prime}}$ with areas $z^{\prime}$ and $z^{\prime\prime}$ and aspect ratios of $\rho^{\prime}$ and $\rho^{\prime\prime}$. Then we need to show that 
\vspace{-2ex}
\begingroup
\setlength{\jot}{2ex} 
\begin{multline}
\label{eq:ApxARge3}
\frac{a+\sqrt{z^{\prime}}(\sqrt{\rho^{\prime}}+1/\sqrt{\rho^{\prime}})+\sqrt{z^{\prime\prime}}(\sqrt{\rho^{\prime\prime}}+1/\sqrt{\rho^{\prime\prime}})}{b+\sqrt{z^{\prime}}+\sqrt{z^{\prime\prime}}} \\ 
\leq \frac{a+\sqrt{z^\prime+z^{\prime\prime}}(\sqrt{\rho^{\prime}+\rho^{\prime\prime}}+1/\sqrt{\rho^{\prime}+\rho^{\prime\prime}})}{b+\sqrt{z^\prime+z^{\prime\prime}}}
\end{multline}
\endgroup 
or equivalently, 
\vspace{-2ex}
\begingroup
\setlength{\jot}{2ex} 
\begin{multline*}
\resizebox{0.95\columnwidth}{!}{%
$a\sqrt{z^\prime+z^{\prime\prime}}+b\sqrt{z^\prime \rho^\prime}+\sqrt{z^\prime \rho^\prime}\sqrt{z^\prime+z^{\prime\prime}} +b\sqrt{\frac{z^\prime}{\rho^\prime}}+\sqrt{\frac{z^\prime}{\rho^\prime}}\sqrt{z^\prime+z^{\prime\prime}}$
} \\
\resizebox{0.92\columnwidth}{!}{%
$+b\sqrt{z^{\prime\prime}\rho^{\prime\prime}}+\sqrt{z^{\prime\prime} \rho^{\prime\prime}}\sqrt{z^\prime+z^{\prime\prime}} 
+b\sqrt{\frac{z^{\prime\prime}}{\rho^{\prime\prime}}}+\sqrt{\frac{z^{\prime\prime}}{\rho^{\prime\prime}}}\sqrt{z^\prime+z^{\prime\prime}} $
} \\
\resizebox{0.95\columnwidth}{!}{%
$ \leq a\sqrt{z^\prime}+a\sqrt{z^{\prime\prime}}+b\sqrt{z^\prime+z^{\prime\prime}}\sqrt{\rho^\prime+\rho^{\prime\prime}}+\sqrt{z^\prime}\sqrt{z^\prime+z^{\prime\prime}}\sqrt{\rho^\prime+\rho^{\prime\prime}} $
} \\
\resizebox{0.98\columnwidth}{!}{%
$+\sqrt{z^{\prime\prime}}\sqrt{z^\prime+z^{\prime\prime}}\sqrt{\rho^\prime+\rho^{\prime\prime}}+b\frac{\sqrt{z^\prime+z^{\prime\prime}}}{\sqrt{\rho^\prime+\rho^{\prime\prime}}}+\sqrt{z^\prime}\frac{\sqrt{z^\prime+z^{\prime\prime}}}{\sqrt{\rho^\prime+\rho^{\prime\prime}}}
+\sqrt{z^{\prime\prime}}\frac{\sqrt{z^\prime+z^{\prime\prime}}}{\sqrt{\rho^\prime+\rho^{\prime\prime}}}\,,$
}
\end{multline*}
\endgroup
where $a$ is the sum of the perimeter of other sub-rectangles inside our main rectangle and $b$ is their corresponding lower bound values.
We will show the following inequalities:
\[
a\sqrt{z^\prime+z^{\prime\prime}}+b\sqrt{\frac{z^{\prime\prime}}{\rho^{\prime\prime}}} \leq a\sqrt{z^\prime}+a\sqrt{z^{\prime\prime}}\,,
\]
\[
b\sqrt{z^\prime\rho^\prime}+b\sqrt{z^{\prime\prime} \rho^{\prime\prime}} \leq b\sqrt{z^\prime+z^{\prime\prime}}\sqrt{\rho^\prime+\rho^{\prime\prime}}\,,
\]
\[
\sqrt{z^\prime+z^{\prime\prime}}\sqrt{z^\prime\rho^\prime}+\sqrt{z^\prime+z^{\prime\prime}}\sqrt{\frac{z^\prime}{\rho^\prime}}\leq \sqrt{z^\prime+z^{\prime\prime}}\sqrt{z^\prime}\sqrt{\rho^\prime+\rho^{\prime\prime}}\,,
\]
\[
\sqrt{z^{\prime\prime}\rho^{\prime\prime}}+\sqrt{\frac{z^{\prime\prime}}{\rho^{\prime\prime}}}\leq \sqrt{z^{\prime\prime}}\sqrt{\rho^\prime+\rho^{\prime\prime}}\,,
\]
and
\[
b\sqrt{\frac{z^\prime}{\rho^\prime}}=b\sqrt{\frac{z^\prime+z^{\prime\prime}}{\rho^\prime+\rho^{\prime\prime}}}
\]
By summing the two sides of these inequalities and by the fact that $\sqrt{z^\prime}\frac{\sqrt{z^\prime+z^{\prime\prime}}}{\sqrt{\rho^\prime+\rho^{\prime\prime}}}+\sqrt{z^{\prime\prime}}\frac{\sqrt{z^\prime+z^{\prime\prime}}}{\sqrt{\rho^\prime+\rho^{\prime\prime}}}$ is positive we will prove \eqref{eq:ApxARge3}.

~\\
First, for $\rho^\prime,\rho^{\prime\prime} \geq 3$ we have
\[
1 \leq \sqrt{\rho^{\prime}}+\sqrt{\rho^{\prime\prime}}-\sqrt{\rho^\prime+\rho^{\prime\prime}} 
\]
Since $a\geq b$ we multiply left hand side by $b$ and right hand side by $a$. We get
\vspace{-1ex}
\begingroup
\setlength{\jot}{1ex} 
\begin{multline*}
b \leq a\sqrt{\rho^{\prime}}+a\sqrt{\rho^{\prime\prime}}-a\sqrt{\rho^\prime+\rho^{\prime\prime}} \; \Rightarrow \\
bh_z \leq ah_z\sqrt{\rho^{\prime}}+ah_z\sqrt{\rho^{\prime\prime}}-ah_z\sqrt{\rho^\prime+\rho^{\prime\prime}} \; \Rightarrow \\
ah_z\sqrt{\rho^\prime+\rho^{\prime\prime}}+bh_z\leq ah_z\sqrt{\rho^{\prime}}+ah_z\sqrt{\rho^{\prime\prime}} \; \Rightarrow \\
a\sqrt{z^\prime+z^{\prime\prime}}+b\sqrt{\frac{z^{\prime\prime}}{\rho^{\prime\prime}}} \leq a\sqrt{z^\prime}+a\sqrt{z^{\prime\prime}}
\end{multline*}
\endgroup

~\\
Second, we need to show that $\sqrt{z^\prime\rho^\prime}+\sqrt{z^{\prime\prime} \rho^{\prime\prime}} \leq \sqrt{z^\prime+z^{\prime\prime}}\sqrt{\rho^\prime+\rho^{\prime\prime}}$. After raising both sides to the power of two, this can be simplified to $2\sqrt{z^\prime z^{\prime\prime}\rho^\prime\rho^{\prime\prime}}\leq z^\prime\rho^{\prime\prime}+z^{\prime\prime}\rho^\prime$, which holds because $(\sqrt{z^\prime\rho^{\prime\prime}}-\sqrt{z^{\prime\prime}\rho^\prime})^2\geq 0$.

~\\
Third and Fourth, we need to show $\sqrt{\rho^\prime+\rho^{\prime\prime}} \geq \sqrt{\rho^\prime}+\sqrt{\frac{1}{\rho^\prime}}$ and $\sqrt{\rho^\prime+\rho^{\prime\prime}} \geq \sqrt{\rho^{\prime\prime}}+\sqrt{\frac{1}{\rho^{\prime\prime}}}$. Both hold for $\rho^\prime, \rho^{\prime\prime} \geq 3$.

~\\
Fifth, for the last one we need to show that $\rho^\prime z^{\prime\prime}=\rho^{\prime\prime}z^\prime$. Replacing $\rho^\prime$ with $w_{z^\prime}/h_{z^\prime}$, $z^\prime$ with $w_{z^\prime}h_{z^\prime}$, and doing the same for $\rho^{\prime\prime}$ and $z^{\prime\prime}$ the equality holds since $h_{z^\prime}=h_{z^{\prime\prime}}$.
This completes the proof of Lemma \ref{lem:ApxARGe3} for $\rho^\prime, \rho^{\prime\prime}\geq 3$.

Now, suppose that $\rho^{\prime\prime} \leq 3$. We know that 
\[
\frac{a+\sqrt{z^\prime}(\sqrt{\rho^\prime}+\frac{1}{\sqrt{\rho^\prime}})}{b+\sqrt{z^\prime}} \leq \frac{a+\sqrt{z^\prime+z^{\prime\prime}}(\sqrt{\rho^\prime+\rho^{\prime\prime}}+\frac{1}{\sqrt{\rho^\prime+\rho^{\prime\prime}}})}{b+\sqrt{z^\prime+z^{\prime\prime}}}\,,
\]
for $\rho^\prime \geq 3$ and $\rho^{\prime\prime} \leq 3$ and $b\leq a \leq (\sqrt{\rho^\prime}+1/\sqrt{\rho^\prime})b$, 
based on the fact that the aspect ratio of $R_{z^{\prime}}$ is less than or equal to the aspect ratio of $R_{z^{\prime}} \cup R_{z^{\prime\prime}}$.
Now, we want to show that inequality \eqref{eq:ApxARge3} holds for this case too. 
\vspace{-2ex}
\begingroup
\setlength{\jot}{2ex} 
\begin{multline*}
\frac{a+\sqrt{z^{\prime}}(\sqrt{\rho^{\prime}}+1/\sqrt{\rho^{\prime}})+\sqrt{z^{\prime\prime}}(\sqrt{\rho^{\prime\prime}}+1/\sqrt{\rho^{\prime\prime}})}{b+\sqrt{z^{\prime}}+\sqrt{z^{\prime\prime}}} \\
\leq \max\{\frac{a+\sqrt{z^\prime}(\sqrt{\rho^\prime}+\frac{1}{\sqrt{\rho^\prime}})}{b+\sqrt{z^\prime}}, 2/\sqrt{3}\} \\
 \leq \frac{a+\sqrt{z^\prime+z^{\prime\prime}}(\sqrt{\rho^\prime+\rho^{\prime\prime}}+\frac{1}{\sqrt{\rho^\prime+\rho^{\prime\prime}}})}{b+\sqrt{z^\prime+z^{\prime\prime}}}
\end{multline*}
\endgroup
Hence, $R_{z1}$ is a simple rectangle and we cannot have more than 1 simple rectangle having width greater than height. The only situation that by dividing $R_2$ with area $z$ into other rectangles we increase the approximation factor is when $R_{z2}$ has aspect ratio greater than $\rho_z$ with width smaller than height. Otherwise, it contributes positively to the approximation factor and decreases it.
\end{proof}
\end{lem}
\begin{cor}
If the algorithm divides $R$ into two rectangles $R_1$ and $R_2$ and $\Area(R_2)=z, \; \width(R_2)= w_z,\; \height(R_2)=h_z$ (with $w_z\geq h_z$), and $\AR(R_2)=\rho_z\geq 3$, which makes the total approximation factor to come above $2/\sqrt{3}$, the maximum approximation factor will be less than 
\[
\frac{a+w_z+1.5z/w_z}{b+2\sqrt{z}}\,,
\]
where $a$ is the sum of the perimeter of other sub-rectangles inside our main rectangle and $b$ is their corresponding lower bound values.
\begin{proof}
In Lemma \ref{lem:ApxARGe3} we showed that in the worst case we do not have more than one rectangle with same direction in a rectangle with aspect ratio greater than 3. So, in the worst case we have a simple rectangle and a rectangle in another direction in each division. We know that in each division to make the perimeter maximized, we should consider the lowest aspect ratio. For example, when we have aspect ratio of 3 the height of the simple rectangle that is created is greater than when the aspect ratio is 4. So, the total perimeter of this sub-rectangles from when we have our first non-forced rectangle is:
$$w_z+\frac{\frac{z}{9}}{w_z}+\frac{\frac{z}{81}}{\frac{z}{w_z}}+ \cdots \leq w_z+ \frac{1.5z}{w_z}$$
Furthermore,
the approximation factor would be less than:
$$\frac{a+w_z+1.5z/w_z}{b+2\sqrt{z_1}+2\sqrt{z_2}+...+2\sqrt{z_k}}\leq \frac{a+w_z+1.5z/w_z}{b+2\sqrt{z}}$$ 
\end{proof}
\end{cor}

\subsubsection{Analysis of the Approximation Factor}
\label{sec:ApxCases}
In this section, we analyze various cases that there is a compound rectangle with aspect ratio greater than 3 and combine the rectangle upper bound and lower bound with other simple sub-rectangles in order to prove the bound of our approximation algorithm.

~\\
~\\
\textbf{Case 1.} Consider a rectangle $R$ with aspect ratio $\rho \leq 3$ that is dissected into two rectangles $R_y$ which is a simple rectangle and $R_t$ which is a compound rectangle having aspect ratio less than 3 consisting a simple rectangle as its left  child ($R_x$) and $R_z$ as its right child that can be either a compound or a simple rectangle and $R_z$ is located on top of $R_x$. For simplicity, suppose that area of $R$ is 1 and its width is greater than its height (see Figure \ref{fig:1_apprx}).
\begin{figure}[!ht]
    \centering
    \includegraphics[scale=0.6]{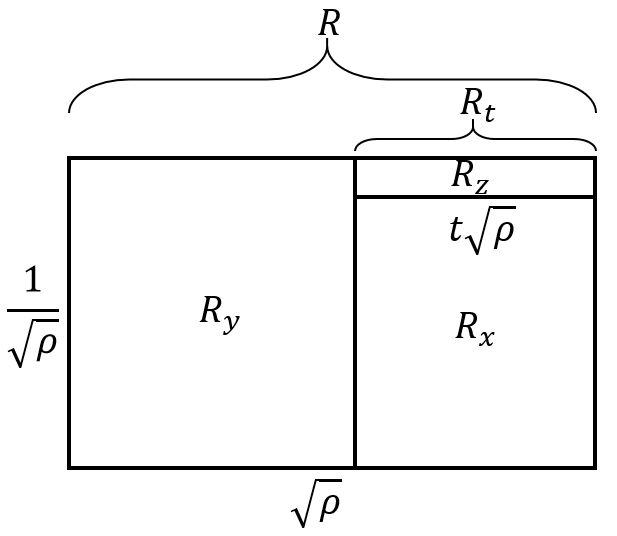}
    \caption{Rectangle instance of Case 1 in the proof of approximation factor.}
    \label{fig:1_apprx}
\end{figure}

\textbf{Case 1.1}
First, we show that if $z \geq 0.01$, the approximation factor is less than 1.19. We need to show that 
$$\max \, \frac{\sqrt{\rho}+t\sqrt{\rho}+2/\sqrt{\rho}+1.5z/\sqrt{\rho}\,t}{2(\sqrt{x}+\sqrt{z}+\sqrt{1-t})} \leq 1.19 $$
s.t.
$$1 \leq \rho \leq 3$$
$$x+z = t$$
$$y+t=1$$
$$\rho>\frac{z}{t^2}$$
$$\rho \geq \frac{1}{3t}$$
$$\rho \leq \frac{1}{t}$$
$$1-t \geq x(=t-z)$$
Where inequalities state the range of rectangle $R, R_t$ and $R_z$ aspect ratios. 
We also have $z \geq 0.01$.

Proof of the bound:\\
Replacing $x$ with $t-z$ we need to show that:
\[
\resizebox{\columnwidth}{!}{%
$\min \, (2.38\sqrt{t-z}+2.38\sqrt{1-t}+2.38\sqrt{z}-\sqrt{\rho}-t\sqrt{\rho}-\frac{2}{\sqrt{\rho}}-\frac{1.5z}{\sqrt{\rho}\,t} )\geq 0$
}
\]
The second derivative of the expression over $z$ is:
\[
\frac{-0.595}{\sqrt{z^3}}-\frac{0.595}{\sqrt{(t-z)^3}} < 0 
\]
Therefore, minimum of expression is either at $z=\max\{0.01,2t-1\}$ or at $z= \frac{\rho t^2}{3}$.\\

$\bm{z=0.01 \; (2t-1 \leq 0.01):}$ We need to show that:
\[
\resizebox{\columnwidth}{!}{%
$2.38\sqrt{t-0.01}+2.38\sqrt{1-t}+0.238-\sqrt{\rho}-t\sqrt{\rho}-\frac{2}{\sqrt{\rho}}-\frac{0.015}{\sqrt{\rho}t} \geq 0$
}
\]
Second derivative over $t$ is always negative. Then, $t=\frac{1}{3\rho}$ or $t=\min\{\frac{1}{\rho}, 0.505\}$.
\begin{itemize}
    \item $t=0.505 :$ We want to show that:
    \[
     \resizebox{0.94\columnwidth}{!}{%
    $4.76\sqrt{0.495}+0.238-\sqrt{\rho}-0.505\sqrt{\rho}-\frac{2}{\sqrt{\rho}}-\frac{0.015}{0.505\sqrt{\rho}} \geq 0$
    }
    \]
    The left hand side is always positive for $1 \leq \rho \leq 1/0.505$.
    \item $t=\frac{1}{\rho} :$ We want to show that:
     \[
     \resizebox{0.94\columnwidth}{!}{%
     $2.38\sqrt{\frac{1}{\rho}-0.01}+2.38\sqrt{1-\frac{1}{\rho}}+0.238-\frac{3}{\sqrt{\rho}}-1.015\sqrt{\rho} \geq 0$
     }
     \]
    The left hand side is always positive for $1/0.505 \leq \rho \leq 3$.
    \item $t=\frac{1}{3\rho}:$ We want to show that:
\[
\resizebox{0.94\columnwidth}{!}{%
$2.38\sqrt{\frac{1}{3\rho}-0.01}+2.38\sqrt{1-\frac{1}{3\rho}}+0.238-\frac{7}{3\sqrt{\rho}}-1.045\sqrt{\rho} \geq 0$\,,
}
\]
which holds for all $1 \leq  \rho \leq 3$. \\
\end{itemize}
$\bm{z=2t-1 \;(0.01 \leq 2t-1):}$ We need to show that:
\[
4.76\sqrt{1-t}+2.38\sqrt{2t-1}-\sqrt{\rho}-t\sqrt{\rho}-\frac{5}{\sqrt{\rho}}+\frac{1.5}{\sqrt{\rho}t} \geq 0
\]
The second derivative over $t$ is:
\[
\frac{d^2}{dt^2}=\frac{-1.19}{\sqrt{(1-t)^3}}+\frac{3}{t^3\sqrt{\rho}}-\frac{2.38}{\sqrt{(2t-1)^3}}\,,
\]
 which is always negative for $\rho=1$ and consequently for all $\rho$'s greater than $1$. Hence minimum occurs in $t=0.505, t=\min(\frac{1}{\rho},\frac{3-3\sqrt{1-\rho/3}}{\rho})$. $\frac{3-3\sqrt{1-\rho/3}}{\rho}$ comes from the fact that $2t-1 \leq \rho t^2/3$. For $t=0.505, z= 2t-1=0.01$ which is investigated before. For other cases we can easily see that the expression is always positive. \\
 
$\bm{z=\frac{\rho t^2}{3}:}$ Replacing $z$ into the expression we need to show that:
\[
\resizebox{\columnwidth}{!}{%
$2.38\sqrt{t-\frac{\rho t^2}{3}}+2.38\sqrt{1-t}+2.38t\sqrt{\frac{\rho}{3}}-\sqrt{\rho}-1.5t\sqrt{\rho}-\frac{2}{\sqrt{\rho}} \geq 0$
}
\]
The minimum of this expression for $1 \leq \rho \leq 3$ and $\frac{1}{9}\leq t\leq 0.6$ is in $\rho=1$ and $t=0.6$ which is 0.56. Hence, it is always positive.

~\\
\textbf{Case 1.2}
Now we want to show that for $z < 0.01$ the approximation factor also holds. To show that, we need some other lemmas. \\
\begin{lem}
For case 1 without considering other rectangles, the approximation factor is bounded by $0.75\sqrt{\rho/2}+\sqrt{1/2\rho}$ for $1 \leq \rho \leq 2$.
\begin{proof}
We need to show that 
\[
\max \frac{\sqrt{\rho}+t\sqrt{\rho}+2/\sqrt{\rho}+1.5z/t\sqrt{\rho}}{2(\sqrt{t-z}+\sqrt{z}+\sqrt{1-t})} \leq 0.75\sqrt{\frac{\rho}{2}}+\sqrt{\frac{1}{2\rho}}
\]
 or
\vspace{-1ex}
\begingroup
\setlength{\jot}{1ex} 
\begin{multline*}
\sqrt{2\rho}(1+t)+\frac{2\sqrt{2}}{\rho}+\frac{3z}{t\sqrt{2\rho}} \\ 
-(1.5\sqrt{\rho}+\frac{2}{\sqrt{\rho}})(\sqrt{t-z}+\sqrt{z}+\sqrt{1-t}) \leq 0
\end{multline*}
\endgroup
s.t.
$$z \geq \max\{0,2t-1\}$$
$$z \leq \rho t^2/3$$
$$\frac{1}{3\rho} \leq t \leq \frac{1}{\rho}$$

$$\frac{d^2}{dz^2}=(1.5\sqrt{\rho}+\frac{2}{\sqrt{\rho}})(\frac{1}{4\sqrt{z^3}}+\frac{1}{4\sqrt{(t-z)^3}}) \geq 0$$
Hence, $z=\max(0,2t-1)$ or $z= \rho t^2/3$. \\

If $2t-1 \leq 0:$ In this case, $z=0$.
The expression would become:
\[
\sqrt{2\rho}(1+t)+\frac{2\sqrt{2}}{\sqrt{\rho}}-(1.5\sqrt{\rho}+\frac{2}{\sqrt{\rho}})(\sqrt{t}+\sqrt{1-t})
\]
The maximum of this expression for $\rho \in \bigg[1,2\bigg]$ and $t \in \bigg[\frac{1}{3\rho},0.5\bigg]$ is in $t=0.5$ which is equal to 0.\\

If $2t-1 \geq 0:$ In this case, $z=2t-1$.
The expression would become:
\[
\resizebox{1\columnwidth}{!}{%
$\sqrt{2\rho}(1+t)+\frac{2\sqrt{2}}{\sqrt{\rho}}+\frac{6t-3}{t\sqrt{2\rho}}-(1.5\sqrt{\rho}+\frac{2}{\sqrt{\rho}})(2\sqrt{1-t}-\sqrt{2t-1})$
}
\]
Again, maximum of this expression for defined intervals of $\rho$ and $t$ is in $t=0.5$ which is equal to 0.\\

If $z=\rho t^2 /3$: The expression would become:
\[
\resizebox{1\columnwidth}{!}{%
$\sqrt{2\rho}(1+t)+\frac{2\sqrt{2}}{\rho}+\frac{t\sqrt{\rho}}{\sqrt{2}}-(1.5\sqrt{\rho}+\frac{2}{\sqrt{\rho}})(\sqrt{t-\frac{\rho t^2}{3}}+\sqrt{\frac{\rho t^2}{3}}+\sqrt{1-t})$
}
\]
Here, $t \in \bigg[\frac{1}{3\rho}, \min\{\frac{1}{\rho},\frac{3-\sqrt{9-3\rho}}{\rho}\}\bigg]$. This comes from substituting $z$ in $y = 1-t \geq t-z$.
Hence, the approximation factor is maximized when $t=0.5$ and $z=0$ for $1 \leq \rho \leq 2$.
\end{proof}
\end{lem}

\begin{lem}
For Case 1 without considering other rectangles, the approximation factor is bounded by $\frac{3+\rho}{2+2\sqrt{\rho-1}}$ for $2 \leq \rho \leq 3$.
\begin{proof}
We need to show that 
\[
\max \frac{\sqrt{\rho}+t\sqrt{\rho}+2/\sqrt{\rho}+1.5z/t\sqrt{\rho}}{2(\sqrt{t-z}+\sqrt{z}+\sqrt{1-t})} \leq \frac{3+\rho}{2+2\sqrt{\rho-1}}
\]
 or
\vspace{-1ex}
\begingroup
\setlength{\jot}{1ex} 
\begin{multline*}
(\frac{2}{\sqrt{\rho}}+\sqrt{\rho}(1+t)+\frac{1.5z}{t\sqrt{\rho}})(\frac{1}{\sqrt{\rho}}+\sqrt{1-\frac{1}{\rho}}) \\ 
-(\sqrt{\rho}+\frac{3}{\sqrt{\rho}})(\sqrt{t-z}+\sqrt{z}+\sqrt{1-t}) \leq 0
\end{multline*}
\endgroup
s.t.
$$z \leq \frac{\rho t^2}{3}$$
$$\frac{1}{3} \leq t \leq \frac{1}{\rho}$$

$$\frac{d^2}{dz^2}=(\sqrt{\rho}+\frac{3}{\sqrt{\rho}})(\frac{1}{4\sqrt{z^3}}+\frac{1}{4\sqrt{(t-z)^3}}) \geq 0$$
Hence, $z=0$ or $z= \rho t^2/3$. \\

$\bm{z=0:}$ In this case,
The expression becomes:
\vspace{-1ex}
\begingroup
\setlength{\jot}{1ex} 
\begin{multline*}
(\sqrt{\rho}(1+t)+\frac{2}{\sqrt{\rho}}+\frac{t\sqrt{\rho}}{2})(\frac{1}{\sqrt{\rho}}+\sqrt{1-\frac{1}{\rho}}) \\
-(\sqrt{\rho}+\frac{3}{\sqrt{\rho}})(t\sqrt{\frac{\rho}{3}}+\sqrt{1-t}+\sqrt{t-\frac{t^2\rho}{3}})
\end{multline*}
\endgroup

The maximum of this expression for $\rho \in \bigg[2,3\bigg]$ and $t \in \bigg[\frac{1}{3},\frac{1}{\rho}\bigg]$ is in $t=\frac{1}{3}$, $\rho =3$ which is equal to 0.\\

$\bm{z=\rho t^2 /3}$: The expression would become:
\[
\resizebox{1\columnwidth}{!}{%
$\sqrt{2\rho}(1+t)+\frac{2\sqrt{2}}{\rho}+\frac{t\sqrt{\rho}}{\sqrt{2}}-(1.5\sqrt{\rho}+\frac{2}{\sqrt{\rho}})(\sqrt{t-\frac{\rho t^2}{3}}+\sqrt{\frac{\rho t^2}{3}}+\sqrt{1-t})$
}
\]
The maximum of this expression occurs in $t=\frac{1}{3}$ and $\rho =\frac{1}{3}$ which is $-0.3853<0$.
\\ Hence, the approximation factor is maximized when $z = 0$ and $t = 1/\rho$ for $2 \leq \rho \leq 3$.
\end{proof}
\end{lem}
For Case 1 if rectangle $R$ in Figure \ref{fig:1_apprx} is not the main rectangle that we intend to partition, It must have a parent and a sibling. We call the sibling $R^\prime$. $R^\prime$ would locate at the bottom/ top of rectangle $R$ or at the left/right side of that. We need to know how many children $R^\prime$ has. First we show that it is either a simple rectangle or has up to four children.
First, remember that we sort the list of areas in descending order. We have $R_y$, $R_x$ and $R_z$ with areas $y$, $x$ and $z$ respectively. We know that $y\geq x \geq z$ and $x$ and $z$ are located at the end of the list of areas and the would sum up and will locate in new ordered list of areas. 
We know that before summing $x+z$ and $y$ together (Before they locate at the end of the list) several double areas could get summed together. We know that they make new areas greater than $x+z$. Making new areas from areas smaller than $x+z$ and $y$ by summing up them in pairs continues until $x+z$ and $y$ be located at the end of the list. Now they will be summed (equal to 1) and relocate in the list. Now in the list there might be one single and several paired areas or just several paired areas after $x+y+z$. Then these area ahead of $x+y+z$ get summed while we know that their summation is greater than $x+y+z=1$. This continues until it is time for $x+y+z$ to get summed with another area, while we know that the current areas in the list  are consisted from at most 4 areas from the original list.  \\
Here we separate the discussion into two cases; when $R^\prime$ is a simple rectangle and when it a compound rectangle.

~\\
\textbf{$\bm{R^{\prime}}$ is a compound rectangle:}
Suppose that $R^\prime$ is a compound rectangle containing 2, 3 or 4 sub-rectangles. 
Moreover, since when $x+z$ and $y$ locate at the end of the list, the area of combinations of two areas which are located before $x+z$ and $y$ is greater than $\max\{2x, y\}$, considering it with $z\leq \frac{x}{2}$ results in $a_i,a_{i-1}\geq x\geq \frac{2}{7}$. Also, the area of $R\prime$ cannot be greater than 2. Because when $y$ and $x+z$ are summed there are two areas in the list ahead of them with area less than $x+y+x=1$ which are children of $R\prime$ (consisting single or double area).
If $R^\prime$ locate at the bottom/top of rectangle $R$ we know that there are at least two sub-rectangle in them each having area greater than $x$ and less than $x+z$. Remember that $z<0.01$ and these two sub-rectangles hence the areas of these two are very close. We claim that these two sub-rectangles together has $AR \leq 3$. It is not hard to show it using lemma 2 since we know that the AR of rectangle $R$ parent is less than 3 and the area of all sub-rectangles inside $R\prime$ are very close in case of 2 and 4 sub-rectangles because all of them are less than $x+z$ and greater than $x$. In the case of 3 sub-rectangles assume that before having 3 sub-areas $a_0$ is the single child of $R\prime$ and $a_1$ and $a_2$ are the ones which are summed together. So, at some point the list was:
$$..., a_0, x+z, y, a_1, a_2$$
Note that the order of $x+z$ and $y$ could be switched.
$$..., a_0, a_1+a_2, x+z, y$$
Again in addition to order of $x+z$ and $y$ the order of $a_0$ and $a_1+a_2$ could be switched. Then the list will be updated to
$$...,x+z+y , a_0, a_1+a_2 \,,$$ and then
$$...,a_0+a_1+a_2,x+z+y \,,$$
where finally $x+z+y$ is getting summed with these three sub-areas. Hence, in the case that $a_0$ is less than $a_1+a_2$, $a_0 \geq \max\{x+z,y\} \geq 0.5$ and $a_1+a_2\leq 1$. Hence the AR of $a_0$ and compound rectangle having $a_1+a_2$ is less than 3. 
In the case that  $a_0$ is greater than $a_1+a_2$, $a_0\leq 1$ and $a_1+a_2 \geq \max\{2x,y\}$ which knowing that $z\leq 0.01$ results in $a_1+a_2\geq 0.66$. Thus again the AR of $a_0$ and compound rectangle having $a_1+a_2$ is less than 3. 
The same result holds for when $R^\prime$ would locate at the right/left side of rectangle $R$, because these two sub-rectangles always have area less than $x+y+z=1$ and hence the AR of $R$ containing $x,y,z$ is less than 3, the AR of the rectangle containing these two sub-rectangles is also less than 3. \\
As we showed what we claimed now based on the fact that the area of two sub-rectangles are very close the AR of the bigger one is always less than 2 and the other one less than 3. 
Moreover, we know that each of them has an area greater than $0.33$.

~\\
\textbf{$\bm{R^{\prime}}$ is a simple rectangle:}
$R^\prime$ is a simple rectangle, its area is greater than $\max\{x+z,y\}$ and hence greater than 0.5.
If it is not a forced simple rectangle the AR of $R$ and $R^\prime$ together must be less than 3, since otherwise,it would be defined as $z$. If $R^\prime$ is at the bottom/top the worst AR happens when $\rho = 1.5$ and $a_{R^\prime} = 0.5$ and the AR is 3. That is because in order to have $R^\prime$ at the top/bottom of $R$ and not at the right/left side of it the height of parent of $R$ and $R^\prime$ should be greater than $\sqrt{\rho}$ which is the width and hence the AR of $R^\prime$ in the case of its width is greater than its height, is $\frac{\sqrt{\rho}}{\sqrt{\rho}-\frac{1}{\sqrt{\rho}}}$. Therefore, the area should be greater than $\sqrt{\rho}(\sqrt{\rho}-\frac{1}{\sqrt{\rho}})=\rho-1$
Together with area greater than 0.5, we know that areas should be greater than $\max\{0.5,\rho-1\}$. On the other hand, since the AR of $R$ and $R^\prime$ together should not be greater than 3, the area of $R^\prime$ should be less than $3\rho-1$. \\
In the situation that $R^\prime$ is at the left/right side of $R$, if the area of $R^\prime$ is $A_R^\prime$ its AR should be $\rho A_R^\prime$ in a way that overall AR of $R$ and $R^\prime$ stay under 3, i.e., $A_{R^\prime}\leq 3/\rho-1$. \\
Now suppose that $R^\prime$ is a forced rectangle, in this case its area is greater than 1 and as a result in approximation factor calculation we can use the lower bound of $w+h$. Clearly, this lower bound reduces the overall approximation factor more than having a rectangle with area $A_R^\prime$ AR of $\rho A_R^\prime$ while the overall AR of $R$ and $R^\prime$ stay under 3.

\begin{lem}
If Rectangle $R$ in Case 1 (Figure \ref{fig:1_apprx}) is not the main rectangle that we want to partition the approximation factor of Algorithm \ref{alg:approxAlg} is 1.203.
\begin{proof}
By lemma 2.8.4 and 2.8.5 we have the approximation factors of rectangle $R$ in case 1 without considering extra rectangles. 
Now we consider sub-rectangles in $R^\prime$ in the calculation of approximation factor. First for the case that $R^\prime$ is not simple. We showed that there are at least two sub-rectangles in $R^\prime$ that has area greater than 0.33 and one of them $AR \leq 2$ and the other $AR \leq 3$. We include them to the approximation factors we had. For $ 1 \leq \rho \leq 2$ it will be:
$$\frac{1.5\sqrt{\rho}+\frac{2}{\sqrt{\rho}}+4\sqrt{0.11}+3\sqrt{0.165}}{2\sqrt{2}+4\sqrt{0.33}}$$
$4\sqrt{0.11}+3\sqrt{0.165}$ comes from putting $wh =0.33$ and $w/h=3$ and $w/h=2$ respectively.
The expression is maximized for $\rho =2$ and the value is 1.1862.
For $ 2 \leq \rho \leq 3$ it will be:
$$\frac{\sqrt{\rho}+\frac{3}{\sqrt{\rho}}+4\sqrt{0.11}+3\sqrt{0.165}}{2(\frac{1}{\sqrt{\rho}}+\sqrt{1-\frac{1}{\rho}})+4\sqrt{0.33}}$$
This is maximized for $\rho =2$ and the value is 1.1863.\\
Now suppose that $R^\prime$ is a simple rectangle at the left/right of $R$ then $0.5 \leq A_{R^\prime}\leq 3/\rho-1$ the approximation factor would be 
$$\frac{1.5\sqrt{\rho}+\frac{2}{\sqrt{\rho}}+A_{R^\prime}\sqrt{\rho}+\frac{1}{\sqrt{\rho}}}{2\sqrt{2}+2\sqrt{A_{R^\prime}}}$$
Note that since $0.5 \leq 3/\rho-1$, obviously $\rho \leq 2$ and we do not need approximation factor for $\rho \geq 2$
The maximum of this expression is 1.2029 for $\rho =1.5$ and  $A_{R^\prime} = 0.5$\\
Now suppose that $R^\prime$ is a simple rectangle at the top/bottom of $R$ then $\max\{0.5,\rho-1\} \leq A_{R^\prime}\leq 3\rho-1$ the approximation factor would be 
$$\frac{1.5\sqrt{\rho}+\frac{2}{\sqrt{\rho}}+\sqrt{\rho}+\frac{A_{R^\prime}}{\sqrt{\rho}}}{2\sqrt{2}+2\sqrt{A_{R^\prime}}}$$
for $1 \leq \rho \leq 2$ and 
$$\frac{\sqrt{\rho}+\frac{3}{\sqrt{\rho}}+\sqrt{\rho}+\frac{A_{R^\prime}}{\sqrt{\rho}}}{2(\frac{1}{\sqrt{\rho}}+\sqrt{1-\frac{1}{\rho}})+2\sqrt{A_{R^\prime}}}$$
for $2 \leq \rho \leq 3$.
The maximum of these expressions on the defined intervals  is 1.2029 for $\rho = 1.5$ and $A_{R^\prime} =0.5$.
\end{proof}
\end{lem}

~\\
\textbf{Case 2}
Consider a rectangle R with aspect ratio $(\rho)$ less than 3 which is dissected into two rectangles $R_y$ which is a simple rectangle and $R_t$ which is a compound rectangle having aspect ratio less than 3 consisting a simple rectangle as its left child $(R_x)$ and $R_z$ as its right child which can be
either a compound or a simple rectangle and $R_z$ is located at the right side of $R_x$. For simplicity, suppose that area of R is 1 and its width is greater than its height (Figure \ref{fig:2_apprx}) 
\begin{figure}[!ht]
    \centering
    \includegraphics[scale=0.5]{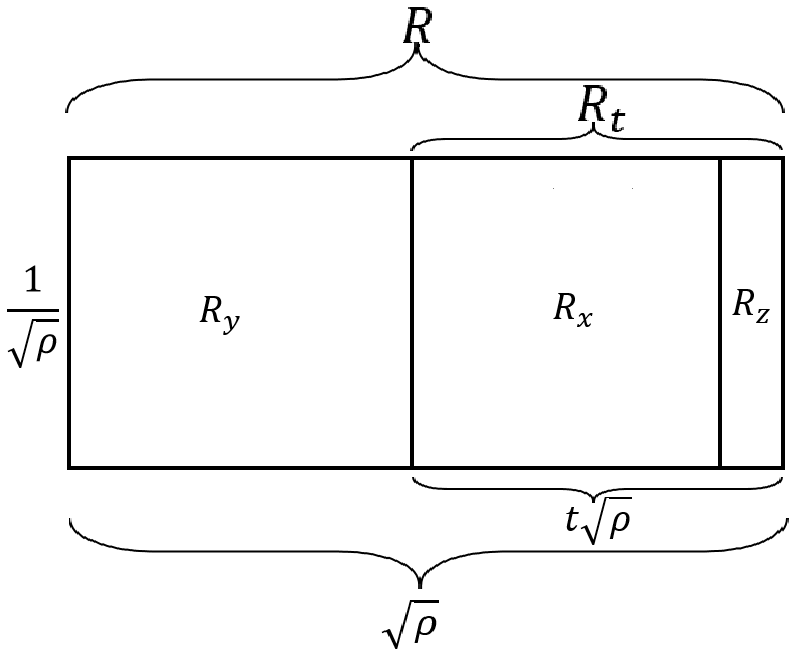}
    \caption{Rectangle instnace of Case 2 in the proof of approximation factor.}
    \label{fig:2_apprx}
\end{figure}
We will show that the worst state of this case is equivalent to the worst case of Case 1 and hence everything we proved for Case 1 also holds here. Therefore, we need to show that the approximation factor in this case is less than $\frac{\frac{3}{\sqrt{\rho}}+\sqrt{\rho}}{2(\sqrt{\frac{1}{\rho}}+\sqrt{1-\frac{1}{\rho}})}$ for $2\leq \rho \leq 3$, i.e.,
$$\frac{\frac{3}{\sqrt{\rho}}+\sqrt{\rho}+1.5\sqrt{\rho}z}{2(\sqrt{t-z}+\sqrt{1-t}+\sqrt{z})} \leq \frac{\frac{3}{\sqrt{\rho}}+\sqrt{\rho}}{2(\sqrt{\frac{1}{\rho}}+\sqrt{1-\frac{1}{\rho}})}$$ 
or 
\vspace{-1ex}
\begingroup
\setlength{\jot}{1ex} 
\begin{multline*}
(\frac{3}{\sqrt{\rho}}+\sqrt{\rho}+1.5\sqrt{\rho}z)(\sqrt{\frac{1}{\rho}}+\sqrt{1-\frac{1}{\rho}}) \\ 
-(\frac{3}{\sqrt{\rho}}+\sqrt{\rho})(\sqrt{t-z}+\sqrt{1-t}+\sqrt{z}) \leq 0
\end{multline*}
\endgroup
s.t.
$$z\leq \frac{1}{3\rho},$$
$$t\geq \frac{1}{\rho},$$
$$z\geq 2t-1$$

$$\frac{d^2}{dz^2} = (\frac{1}{4\sqrt{(t-z)^3}}+\frac{1}{4\sqrt{z^3}})(\frac{3}{\sqrt{\rho}}+\sqrt{\rho}) \geq 0$$
Hence, $z =\max\{0,2t-1\}$ or $z=\frac{1}{3\rho}$.

~\\
If $z=0:$ The expression will be $$(\frac{3}{\sqrt{\rho}}+\sqrt{\rho})(\frac{1}{\sqrt{\rho}}+\sqrt{1-\frac{1}{\rho}}-\sqrt{t}-\sqrt{1-t})$$
Since $t\geq \frac{1}{\rho}$, the expression is always negative.

~\\
If $z=2t-1:$ The expression will be
\vspace{-1ex}
\begingroup
\setlength{\jot}{1ex} 
\begin{multline*}
(\frac{3}{\sqrt{\rho}}+\sqrt{\rho}+1.5\sqrt{\rho}(2t-1))(\frac{1}{\sqrt{\rho}}+\sqrt{1-\frac{1}{p}}) \\ 
-(\frac{3}{\sqrt{\rho}}+\sqrt{\rho})(\sqrt{2t-1}+2\sqrt{1-t})
\end{multline*}
\endgroup
Clearly, $\frac{d^2}{dt^2}\geq 0$ and the maximum value of the expression is in $t=0.5$ or $t=0.6$ (because $t\leq 0.6$), In $t=0.5$ the maximum value of expression is 0 and for $t=0.6$ it is -0.378.
\\If $z=\frac{1}{3\rho}:$
The expression will be
\vspace{-1ex}
\begingroup
\setlength{\jot}{1ex} 
\begin{multline*}
(\frac{3}{\sqrt{\rho}}+\sqrt{\rho}+\frac{1.5}{\sqrt{3\rho}})(\frac{1}{\sqrt{\rho}}+\sqrt{1-\frac{1}{p}}) \\
-(\frac{3}{\sqrt{\rho}}+\sqrt{\rho})(\sqrt{\frac{1}{3\rho}}+\sqrt{t-\frac{1}{3\rho}}+\sqrt{1-t})
\end{multline*}
\endgroup
Since $\frac{d^2}{dt^2}\geq 0$, either $t=\frac{1}{\rho}$ or $t=\frac{1}{2}+\frac{1}{6\rho}$ which has a maximum of $-0.1417$ in $\rho=3$. \\

If $\rho\leq 2$ we need to show that:
$$\frac{\frac{3}{\sqrt{\rho}}+\sqrt{\rho}+1.5\sqrt{\rho}z}{2(\sqrt{t-z}+\sqrt{1-t}+\sqrt{z})} \leq 0.75\sqrt{\frac{\rho}{2}}+\sqrt{\frac{1}{2\rho}}$$
or 
\vspace{-1ex}
\begingroup
\setlength{\jot}{1ex} 
\begin{multline*}
\sqrt{2}(\frac{3}{\sqrt{\rho}}+\sqrt{\rho}+.5z\sqrt{\rho}) \\ 
-(1.5\sqrt{\rho}+\frac{2}{\sqrt{\rho}})(\sqrt{z}+\sqrt{t-z}+\sqrt{1-t})\leq 0
\end{multline*}
\endgroup
s.t $$z \geq \max\{0,2t-1\}$$
$$z \leq 1/3\rho$$
$$\frac{1}{\rho} \leq t \leq 0.5+0.5z$$
\\Since $\frac{d^2}{dz^2}\geq 0$, either $z=\max\{0,2t-1\}$ or $z=\frac{1}{3\rho}$.

~\\
If $z=0$: The only valid value for $\rho$ is 2 and it is similar to previous analysis.

~\\
If $z=2t-1:$ 
\[
\resizebox{1\columnwidth}{!}{%
$\sqrt{2}(\frac{3}{\sqrt{\rho}}-0.5\sqrt{\rho}+3t\sqrt{\rho})-(1.5\sqrt{\rho}+\frac{2}{\sqrt{\rho}})(2\sqrt{1-t}+\sqrt{2t-1})$
}
\]
Clearly, $\frac{d^2}{dt^2}\geq 0$ and the maximum value of the expression is in $t=0.5$ or $t=0.6$ (because $t\leq 0.6$), In $t=0.5$ the maximum value of expression is 0 and for $t=0.6$ it is -0.45.
\\If $z=\frac{1}{3\rho}:$ 
The expression would be
\[
\resizebox{1\columnwidth}{!}{%
$\sqrt{2}(\frac{3}{\sqrt{\rho}}+\sqrt{\rho}+\frac{1}{2\sqrt{\rho}})-(\frac{2}{\sqrt{\rho}}+1.5\sqrt{\rho})(\sqrt{\frac{1}{3\rho}}+\sqrt{t-\frac{1}{3\rho}}+\sqrt{1-t})$
}
\]
Since $\frac{d^2}{dt^2}\geq 0$, either $t=\frac{1}{\rho}$ or $t=\frac{1}{2}+\frac{1}{6\rho}$ which have a maximum of $-0.13$ in $\rho=5/3$. Note that $\rho \geq 5/3$. 

~\\
~\\
\textbf{Case 3:}
Consider a rectangle R with aspect ratio $(\rho)$ less than 3 which is dissected into two rectangles $R_y$ which is a compound rectangle consisting rectangles $y_1$ and $y_2$ divided horizontally and $R_t$ which is a compound rectangle having aspect ratio less than 3 consisting two rectangles $R_x$ and $R_z$ (such as case 1) divided horizontally. For simplicity, suppose that area of $R$ is 1 and its width is greater than its height (Figure \ref{fig:3_apprx}).
\begin{figure}[!ht]
    \centering
    \includegraphics[scale=0.6]{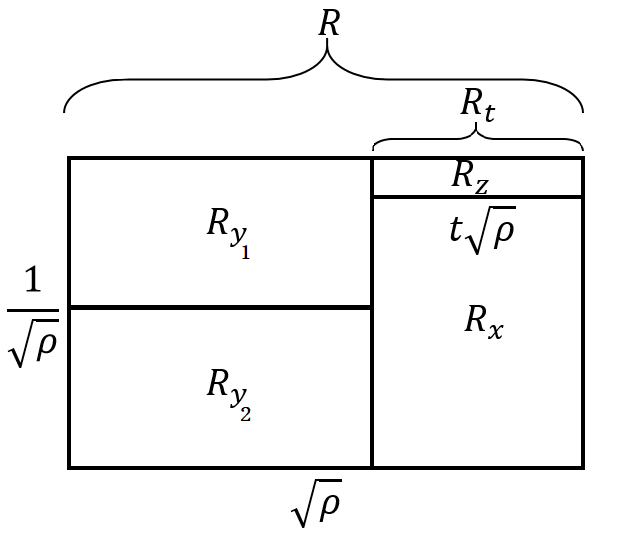}
    \caption{Rectangle instance of Case 3 in the proof of approximation factor.}
    \label{fig:3_apprx}
\end{figure}
For this case to be generated, at some point the list is like $$..., y_1, y_2, x, z,$$
then
$$..., x+z, y_1, y_2$$
and
$$..., y_1+y_2, x+z$$
Therefore, we know that $x+z \geq y_1\geq y_2 \geq x$
$$y_1+y_2 \geq 2x\Rightarrow 2(t-z)\leq 1-t \Rightarrow z \geq 1.5t-0.5$$
Moreover,
$$\frac{\rho t^2}{z} \geq 3 \Rightarrow z \leq \frac{\rho t^2}{3}$$
We also have
$$\frac{1}{\sqrt{\rho}}\geq \sqrt{\rho}(1-t) \Rightarrow t \geq 1-\frac{1}{\rho}$$
On the other hand,
$$t\sqrt{\rho} \leq \frac{1}{\sqrt{\rho}} \Rightarrow t\leq \frac{1}{\rho}$$
We show that the approximation factor for this case is less than 1.2.
$$\frac{2\sqrt{\rho}+\frac{2}{\sqrt{\rho}}+\frac{1.5z}{t\sqrt{\rho}}}{2(\sqrt{t-z}+\sqrt{z}+\sqrt{y_1}+\sqrt{y_2})} \leq 1.2$$
We know that $y_1+y_2=1-t$. In this situation the minimum of them occurs when $y_1$ and $y_2$ are extremely different from each other. Thus, here that we have $y_2\geq x$ the minimum is in $y_2=x$ and $y_1= 1-t-y_2= 1-t-t+z=1-2t+z$ 
So we need to show that
$$2\sqrt{\rho} +\frac{2}{\rho}+\frac{1.5z}{t\sqrt{\rho}}-2.4(2\sqrt{t-z}+\sqrt{z}+\sqrt{1-2t+z}) \leq 0$$
$$\frac{d^2}{d\rho^2} = -\frac{1}{2\sqrt{\rho^3}}+\frac{3}{2\rho^2\sqrt{\rho}}+\frac{9z}{8t\rho^2\sqrt{\rho}} \geq 0$$
Hence, either $\rho = 1$ or $\rho = \min\{\frac{1}{t}, \frac{1}{1-t}\}= \frac{1}{1-t}$.

~\\
If $\rho=1:$ The expression will be 
$$4+\frac{1.5z}{t}-2.4(2\sqrt{t-z}+\sqrt{z}+\sqrt{1-2t+z})$$
For this expression we have $\frac{d^2}{dz^2} \geq 0$
and the maximum is either $z=1.5t-0.5$ or $z=\frac{\rho t^2}{3}$.

~\\
If $z= 1.5t-0.5:$ Replacing $z$ with $1.5t-0.5$ for $\frac{1}{3}\geq t,$ we have the maximum in $t=\frac{1}{3}$ which is -0.1569.

~\\
If $z=\frac{\rho t^2}{3}:$ Replacing $z$ with $z=\frac{\rho t^2}{3}$ for $t\leq 0.36255$ the minimum of expression is in the extreme point of $t$ and is -0.3859. Note that the interval for $t$ in this case is defined by replacing $z$ by $\frac{\rho t^2}{3}$ in $z \geq 1.5t-0.5$.

~\\
Now if $\rho=\frac{1}{1-t}$: The expression would become
\[
\resizebox{1\columnwidth}{!}{%
$\frac{2}{\sqrt{1-t}}+2\sqrt{1-t}+\frac{1.5z\sqrt{1-t}}{t}-2.4(2\sqrt{t-z}+\sqrt{z}+\sqrt{1-2t+z})$
}
\]
Second derivative of this over $z$ is positive and the maximum of expression would be on the extreme points, $z=1.5t-0.5$ or  $z=\frac{\rho t^2}{3}=\frac{t^2}{3-3t}$.
Replacing $z$ with both of these values results in the maximum of expression in $t=\frac{1}{3}$ with negative values -0.0744 and -0.3055 respectively.

~\\
~\\
\textbf{Case 4}
This case is similar to case 3 with the difference that $y_1$ and $y_2$ are divided with a vertical line. (Figure \ref{fig:4_apprx})
\begin{figure}[!ht]
    \centering
    \includegraphics[scale=0.5]{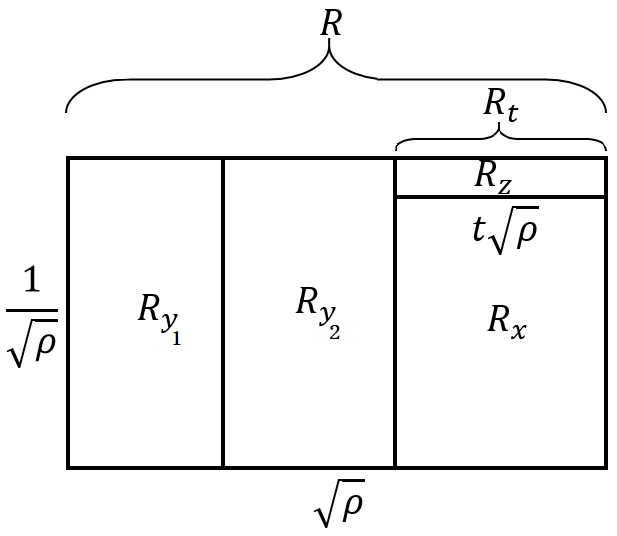}
    \caption{Rectangle instance of Case 4 in the proof of approximation factor.}
    \label{fig:4_apprx}
\end{figure}
So, we need to show that:
$$\frac{\frac{3}{\sqrt{\rho}}+\sqrt{\rho}+t\sqrt{\rho}+\frac{1.5z}{t\sqrt{\rho}}}{2(\sqrt{z}+2\sqrt{t-z}+\sqrt{1-2t+z})}\leq 1.2$$ or
\[
\resizebox{1\columnwidth}{!}{%
$\frac{3}{\sqrt{\rho}}+\sqrt{\rho}+t\sqrt{\rho}+\frac{1.5z}{t\sqrt{\rho}}-2.4(\sqrt{z}+2\sqrt{t-z}+\sqrt{1-2t+z}) \leq 0$
}
\]
s.t. 
$$1-t \geq \frac{1}{\rho} \geq t$$ and
$$1.5t-0.5 \leq z \leq \frac{\rho t^2}{3}$$
Since we know that
$$\frac{d^2}{d\rho^2}=\frac{1}{4\rho^2\sqrt{\rho}}(-\rho-t\rho+9+\frac{4.5z}{t}) \geq 0$$,
we will have $\rho=\frac{1}{t}$ or $\rho=\frac{1}{1-t}$.

~\\
If $\rho =\frac{1}{t}$: The expression will be
$$\frac{1}{\sqrt{t}}+4\sqrt{t}+\frac{1.5z}{\sqrt{t}}-2.4(\sqrt{z}+2\sqrt{t-z}+\sqrt{1-2t+z})$$

Since second derivative of the expression over $z$ is always positive, $z=\frac{\rho t^2}{3}$ or $z=1.5t-0.5$. 

~\\
If $ z=\frac{\rho t^2}{3} = \frac{t}{3}:$ The expression will be 
$$\frac{1}{\sqrt{t}}+4.5\sqrt{t}-2.4(\sqrt{\frac{t}{3}}+2\sqrt{\frac{2t}{3}}+\sqrt{1-\frac{5t}{3}})$$
This expression is negative for any number in $[0.17,0.54]$ which includes $t$.

~\\ 
If $z = 1.5t-0.5$ the expression is:
$$\frac{0.25}{\sqrt{t}}+6.25\sqrt{t}-2.4(\sqrt{1.5t-0.5}+3\sqrt{0.5-0.5t})$$ which is negative for $[1/3,0.54].$

~\\
If $\rho =\frac{1}{1-t}$: The expression will be
$$\frac{1}{\sqrt{t}}+4\sqrt{t}+\frac{1.5z}{\sqrt{t}}-2.4(\sqrt{z}+2\sqrt{t-z}+\sqrt{1-2t+z})$$
Similar to previous expression replacing $\rho = \frac{1}{1-t}$ and taking derivative over $z$, the value of derivative is always positive which results in $z=\frac{\rho t^2}{3}$ or $z=1.5t-0.5$. Replacing them in the expression again results in only negative values for expression.

~\\
~\\
\textbf{Case 5} Consider a rectangle $R$ with aspect ratio ($\rho$) greater than 3 which is dissected into two rectangles $R_y$ which is a simple rectangle and $R_t$ which is a compound rectangle having aspect ratio less than 3 consisting a simple rectangle as its left child ($R_x$) and $R_z$ as its right child which can be either a compound or a simple rectangle and $R_z$ is located on top of $R_x$. For simplicity, suppose that area of $R$ is 1 and its width is greater than its height (Figure \ref{fig:5_apprx}).
\begin{figure}[!ht]
    \centering
    \includegraphics[scale=0.5]{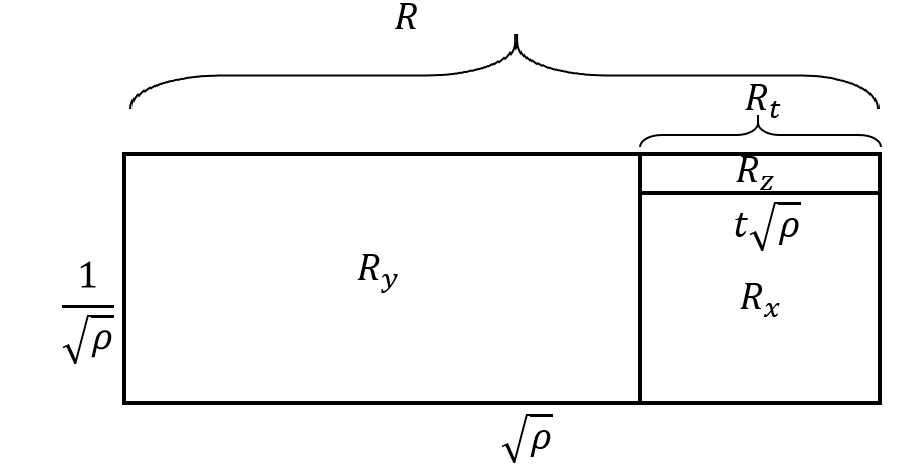}
    \caption{Rectangle instnace of Case 5 in the proof of approximation factor.}
    \label{fig:5_apprx}
\end{figure}
~\\
In this case, where we have $\rho\geq 3$, either $R$ should have belonged to $z$ or $R$ is part of $R_0$. If former happens we are done, but if latter happens we need to show that the approximation factor is less than $1.2$ or
$$\frac{\sqrt{\rho}+t\sqrt{\rho}+\frac{2}{\sqrt{\rho}}+\frac{1.5z}{t\sqrt{\rho}}}{2(\sqrt{z}+\sqrt{t-z})+\sqrt{\rho}(1-t)+\frac{1}{\sqrt{\rho}}} \leq 1.2$$ 
or 
\[
\resizebox{1\columnwidth}{!}{%
$\sqrt{\rho}+t\sqrt{\rho}+\frac{2}{\sqrt{\rho}}+\frac{1.5z}{t\sqrt{\rho}} - 1.2(2(\sqrt{z}+\sqrt{t-z})+\sqrt{\rho}(1-t)+\frac{1}{\rho}) \leq 0$
}
\]
subject ot the constraints be fulfilled.
Since second derivative over $z$ is always positive, $z=0$ or $z=\frac{\rho t^2}{3}$.

~\\
If $z = 0$: The expression becomes
$$\sqrt{\rho}+t\sqrt{\rho}+\frac{2}{\sqrt{\rho}}- 1.2(2\sqrt{t}+\sqrt{\rho}(1-t)+\frac{1}{\sqrt{\rho}})$$
Here, second derivative over $t$ is always positive. Hence, $t=\frac{1}{\rho}$ or $t=\frac{1}{3\rho}$.

~\\If $t=\frac{1}{\rho}$: 
$$\frac{0.6}{\sqrt{\rho}}-0.2\sqrt{\rho}$$
For $\rho \geq 3$ maximum of this expression is equal to 0 at $\rho=3$.

~\\
If $t=\frac{1}{3\rho}$: 
$$\frac{4.6}{3\sqrt{\rho}}-0.2\sqrt{\rho}-\frac{2.4}{\sqrt{3\rho}}$$
For $\rho \geq 3$ maximum of this expression is equal to -0.26 at $\rho=3$.

~\\
If $z = \frac{\rho t^2}{3}$: The expression will be
\[
\resizebox{1\columnwidth}{!}{%
$\sqrt{\rho}+1.5t\sqrt{\rho}+\frac{2}{\sqrt{\rho}} - 1.2(\frac{2t\sqrt{\rho}}{\sqrt{3}}+2\sqrt{t-\frac{\rho t^2}{3}}+\sqrt{\rho}(1-t)+\frac{1}{\sqrt{\rho}})$
}
\]
The value of this expression on defined intervals is on $\rho=3$ which is -0.38.

~\\
~\\
\textbf{Case 6} Consider a rectangle $R$ with aspect ratio ($\rho$) greater than 3 which is dissected into two rectangles $R_y$ which is a simple rectangle and $R_t$ which is a compound rectangle having aspect ratio less than 3 consisting a simple rectangle as its left child ($R_x$) and $R_z$ as its right child which can be either a compound or a simple rectangle and $R_z$ is located at the right side of $R_x$. For simplicity, suppose that area of $R$ is 1 and its width is greater than its height (Figure \ref{fig:6_apprx}).
\begin{figure}[!ht]
    \centering
    \includegraphics[scale=0.5]{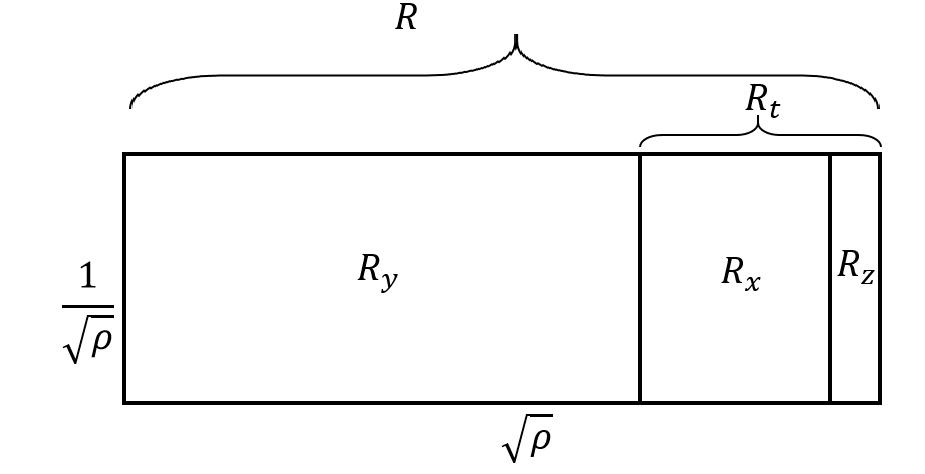}
    \caption{Rectangle instance of Case 6 in the proof of approximation factor.}
    \label{fig:6_apprx}
\end{figure}
~\\
Similar to Case 5, we have $\rho\geq 3$ and, either $R$ should have belonged to $z$ or $R$ is part of $R_0$. If former happens we are done, but if latter happens, $R_y$ would be a part of main rectangle. The proves for case 2 also holds here and the worst case of this case is similar to worst case of case 5 which we showed the approximation factor is less than 1.2.

~\\
~\\
\textbf{Case 7} Suppose that $R_0$ is a rectangle with aspect ratio ($\rho$) which using our approximation algorithm is dissected into two rectangles $R_x$ which is a simple rectangle and $R_z$ which is either a simple or compound rectangle and $R_x$ and $R_z$ are divided by a vertical line. For simplicity, suppose that area of $R$ is 1 and its width is greater than its height (Figure \ref{fig:7_apprx} ).
\begin{figure}[!ht]
    \centering
    \includegraphics[scale=0.5]{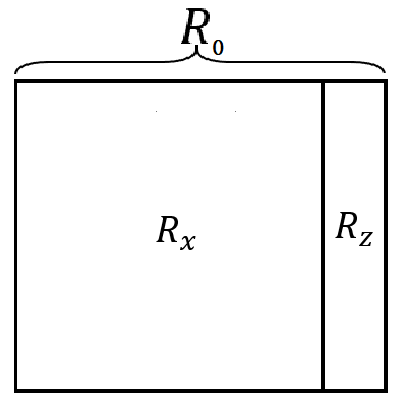}
    \caption{Rectangle instance of Case 7 in the proof of approximation factor.}
    \label{fig:7_apprx}
\end{figure}
~\\
In this case, if width of $R_x$ is greater than its height it would be a forced rectangle. On the other hand most part of $R_z$ is composed of either forced rectangles or rectangles with aspect ratio less than 3. The reason behind this is that either height of those rectangles are greater than their width or unless the very end of the area list the rectangles that are generated inside $z$ in the middle steps have $AR\leq 3$. 
Hence, in the worst case we can consider the big portion of $z$ has $AR=3$ and a small portion has $AR>3$ which is not a forced rectangle and may include some sub-rectangles.\\
To do that we separate these two parts. The small portion cannot have an area greater than $\frac{z^2\rho}{3}$ since its AR should be greater than 3 (Figure \ref{fig:8_apprx}).
\begin{figure}
    \centering
    \includegraphics[scale=0.5]{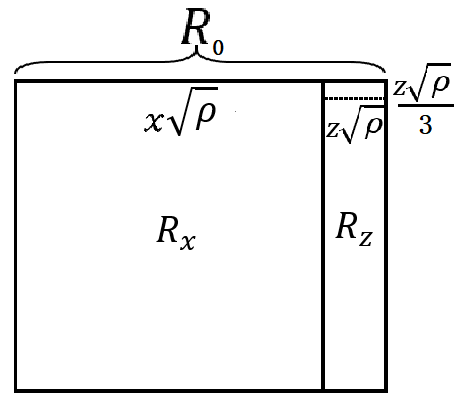}
    \caption{Rectangle mentioned in Case 7 in the proof of approximation factor}
    \label{fig:8_apprx}
\end{figure}

As a result the rest area of $R_z$ is greater than $z-\frac{z^2\rho}{3}$. Considering all this part consisting sub-rectangles with $AR=3$, we can ignore them because they only have positive contribution in compensating for the small section of $R_z$ with aspect ratio more than 3. The upper bound of smaller area is $z\sqrt{\rho}+z\frac{\sqrt{\rho}}{3}+1.5\frac{z\sqrt{\rho}}{3}=\frac{5.5z\sqrt{\rho}}{3}$.
First assume that width of $R_x$ is greater than its height. In this situation It is enough to show that:
$$\frac{x\sqrt{\rho}+\frac{1}{\sqrt{\rho}}+\frac{5.5z\sqrt{\rho}}{3}}{x\sqrt{\rho}+\frac{1}{\sqrt{\rho}}+\frac{2z\sqrt{\rho}}{\sqrt{3}}}\leq 1.2$$
s.t. $$z\leq \frac{1}{3\rho},$$
$$z\leq 1-\frac{1}{\rho}$$

First constraint refers to the lower bound of 3 for aspect ratio of $R_z$ and second constraint ensures that $R_x$ is a forced rectangle. 
So, we need to show that
\[
\resizebox{1\columnwidth}{!}{%
$(1-z)\sqrt{\rho}+\frac{1}{\sqrt{\rho}}+\frac{5.5z\sqrt{\rho}}{3}-1.2((1-z)\sqrt{\rho}+\frac{1}{\sqrt{\rho}}+\frac{2z\sqrt{\rho}}{\sqrt{3}})\leq 0$
}
\]

The maximum of this expression is -0.2171 in $\rho=4/3$ and $z= 0.25$.

~\\
Now assume that width of $R_x$ is less than its height. In this situation it is enough to show that:
$$\frac{x\sqrt{\rho}+\frac{1}{\sqrt{\rho}}+\frac{5.5z\sqrt{\rho}}{3}}{2\sqrt{x}+\frac{2z\sqrt{\rho}}{\sqrt{3}}}\leq 1.2$$
s.t. $$z\leq \frac{1}{3\rho},$$
$$z> 1-\frac{1}{\rho}$$
We want to show that
$$(1-z)\sqrt{\rho}+\frac{1}{\sqrt{\rho}}+\frac{5.5z\sqrt{\rho}}{3}-1.2(2\sqrt{1-z}+\frac{2z\sqrt{\rho}}{\sqrt{3}})\leq 0$$
The maximum of this expression is -0.2171 in $\rho=4/3$ and $z= 0.25$.

~\\
~\\
\textbf{Case 8} This case is similar to Case 1 (Figure \ref{fig:1_apprx}) with the only difference that rectangle $R$ is the main rectangle that we want to partition($R_0$). For this case, either $R_y$ is a forced rectangle which causes $R_x$ to be a forced rectangle too or $R_y$ is not a forced rectangle. Similar to the analysis we did for case 7, when $R_y$ and $R_x$ are forced rectangle the approximation factor of whole $R$ is very low because only a very small part of $R_z$ may not be either a forced rectangle or a rectangle with AR less than 3.
Thus we analyze for the situation that $R_y$ is not a forced rectangle. \\
Remember that a big proportion of $R_z$ is consist of sub-rectangles that either have width greater than height or has aspect ratio less than 3. The sub-rectangles that have larger width than height, we know that every other layout cause them to have bigger aspect ratios. (Figure \ref{fig:9_apprx})
Hence again we disregard them and also the ones that have aspect ratio less than 3 in our calculation. The small part of $R_z$ that can be improved in other layouts has the maximum area of $\frac{z^2}{t\sqrt{3\rho}}$. Because it has $AR > 3$ and height of $\frac{z}{\sqrt{\rho}t}$
and show that:
\begin{figure}[tbp]%
    \centering%
    \begin{subfigure}[b]{.4\textwidth}%
        \includegraphics[width=\textwidth]{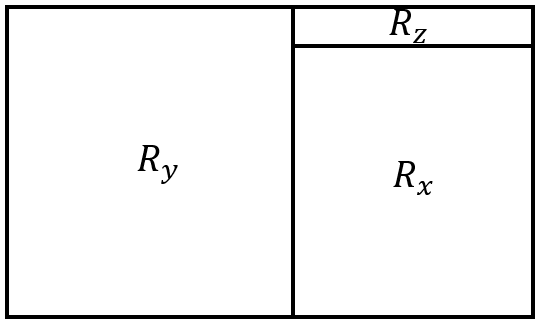}%
    \end{subfigure}%
    %
    \hspace{0.05\textwidth}
    \begin{subfigure}[b]{.4\textwidth}%
        \includegraphics[width=\textwidth]{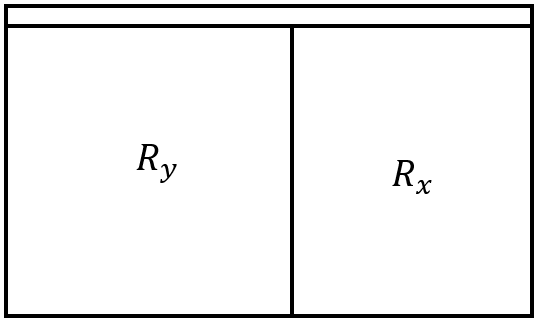}%
    \end{subfigure}%
    ~\\%
    ~\\%
    \begin{subfigure}[b]{.4\textwidth}%
        \includegraphics[width=\textwidth]{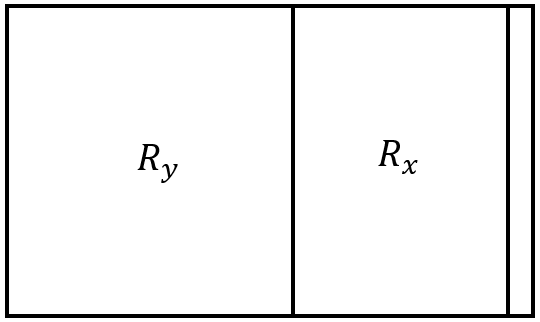}%
    \end{subfigure}%
    %
   \hspace{0.05\textwidth}
    \begin{subfigure}[b]{.4\textwidth}%
        \includegraphics[width=\textwidth]{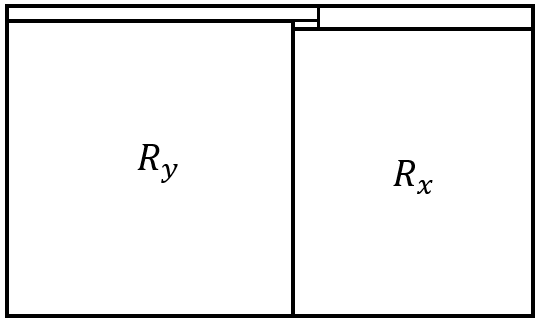}%
    \end{subfigure}%
    \subfigsCaption{\protect
        Four different sub-cases of Case 8 in the proof of approximation factor.%
    }%
    \label{fig:9_apprx}%
\end{figure}
$$\frac{\sqrt{\rho}+\frac{2}{\rho}+\frac{z}{2t\sqrt{\rho}}}{2(\sqrt{y}+\sqrt{x}+\frac{z}{e\sqrt{3\rho}})} \leq 1.2$$
s.t.
$$t\geq 1-\frac{1}{\rho}$$
$$z < \frac{\rho t^2}{3}$$
$$t\geq \frac{1}{3\rho}$$
$$t\leq 0.5+0.5z$$
$$t\leq \frac{1}{\rho}$$
So, we need to prove
$$\sqrt{\rho}+\frac{2}{\sqrt{\rho}}+\frac{z}{2t\sqrt{\rho}}-2.4(\sqrt{1-t}+\sqrt{t-z}+\frac{z}{t\sqrt{3\rho}}) \leq 0$$
Since second derivative of the expression over $z$ is always positive, $z=0$ or $z=\frac{\rho t^2}{3}$.

~\\
If $z=0$ the expression will be
$$\sqrt{\rho}+\frac{2}{\sqrt{\rho}}-2.4(\sqrt{1-t}+\sqrt{t})$$
From this we have either $t=\max\{\frac{1}{3\rho},1-\frac{1}{\rho}\}$ or $t=\min\{\frac{1}{\rho},0.5\}$.

~\\
If $t = \frac{1}{\rho} (\frac{1}{\rho}\leq 0.5)$: Since $1-\frac{1}{\rho} \leq \frac{1}{\rho}$, $\rho$ is bounded above by 2. So, $\rho=2$ is the only possible value for $\rho$ which makes the expression equal to -0.565.

~\\
If $t = \frac{1}{3\rho} (\frac{1}{3\rho}\geq 1-\frac{1}{\rho})$ The expression will be
$$\sqrt{\rho}+\frac{2}{\sqrt{\rho}}-2.4\sqrt{1-\frac{1}{3\rho}}-2.4\frac{1}{\sqrt{3\rho}}$$
The maximum of this expression for $1\leq \rho \leq 4/3$ would be on $\rho=1$ which is -0.34.

~\\
If $t = 1-\frac{1}{\rho} (\frac{1}{3\rho}\leq 1-\frac{1}{\rho})$: The expression will be
$$\sqrt{\rho}+\frac{2}{\sqrt{\rho}}-2.4\sqrt{1-\frac{1}{\rho}}-\frac{2.4}{\sqrt{\rho}}$$
The maximum of this expression for $4/3\leq \rho \leq 2$ would be on $\rho=4/3$ which is -0.54.

~\\
If $z=\frac{\rho t^2}{3}$ the expression becomes
$$\sqrt{\rho}+\frac{2}{\sqrt{\rho}}+\frac{t\sqrt{\rho}}{6}-2.4(\sqrt{1-t}+\sqrt{t-\frac{\rho t^2}{3}}+\frac{t\sqrt{\rho}}{3\sqrt{3}})$$
The maximum of this is $-0.357$ on $t=0.6$ and $\rho =1$.

~\\
~\\
\textbf{Case 9}
This case is similar to Case 2 (Figure \ref{fig:2_apprx}) with the only difference that rectangle R is the main rectangle that we want to partition($R_0$). The worst case of this case is similar to Case 8 and proves that the claim holds here too.

~\\
The following theorem summarizes our results.
\begin{thm}
Given a rectangle $R$ with $\Area(R)=A$ and a list of areas $A_1,...,A_n$ with $\sum_{i=1}^n A_i = A$, Algorithm \ref{alg:approxAlg} constructs a partition of $R$ into $n$ sub-rectangles with areas $A_1,...,A_n$ that is a factor 1.203-approximation for the minimum total partition perimeter.
\end{thm}

\section{Conclusion}
\label{sec:conclusion}
We developed a $1.203$-approximation algorithm for the problem of partitioning a rectangle according to list of areas for the sub-rectangles with the objective of minimizing total perimeter of sub-rectangles. This improves the approximation factor of \cite{nagamochi2007approximation}. 
An interesting direction for future research is to modify this algorithm for partitioning a polygon into polygonal sub-regions of given areas. Such results could then be useful in designing ad-hoc networks of geographic resources similar to that of \cite{carlsson2016geometric}. Moreover, equitable partitioning of geographic regions in a way that the distribution of resources among the sub-regions is as close as to one another has various applications in logistics \cite{behrooziCarlsson2020Springer}. The bundling step of the algorithm could be revised to bundle a number of geographic resources, instead of areas, to allocate resources to sub-regions according to a given quota.

\section*{Acknowledgments}
The authors gratefully acknowledge support from a Tier-1 grant from Northeastern University.

\bibliographystyle{IEEEtran}
\bibliography{IEEEabrv,Approximate-partitioning.bib}

\end{document}